\documentclass[12pt]{amsart}

\usepackage{amsmath}
\usepackage{amssymb}  % defines \nmid, for example
\usepackage{latexsym} % for \Box
\usepackage{comment}
\usepackage{url}
\usepackage{fullpage,url,amssymb,amsmath,amsthm,amsfonts,mathrsfs}
\usepackage[usenames,dvipsnames]{color}
\usepackage[pagebackref = true, colorlinks = true, linkcolor = blue, citecolor = Green]{hyperref}
\usepackage[alphabetic,lite]{amsrefs}
\usepackage{enumitem}
\usepackage{amscd}   % for commutative diagrams
\usepackage[all, cmtip]{xy} % for complicated commutative diagrams
\usepackage{xfrac}
\usepackage[T1]{fontenc}
\usepackage{subfiles}
\usepackage{colonequals}

\numberwithin{equation}{section}
\usepackage[all]{xy}
\usepackage{fullpage}
\usepackage{color} 
\def\act#1#2%
  {\mathop{}%
   \mathopen{\vphantom{#2}}^{#1}%
   \kern-3\scriptspace%
   #2}

\newcommand{\Z}{{\mathbb Z}}

\newcommand{\Q}{{\mathbb Q}}

\newcommand{\C}{{\mathbb C}}
\newcommand{\F}{{\mathbb F}}

\newcommand{\PP}{{\mathbb P}}

\newcommand{\fm}{{\mathfrak m}}

\newcommand{\isom}{\simeq}

\DeclareMathOperator{\Tr}{Tr}

\DeclareMathOperator{\Gal}{Gal}

\DeclareMathOperator{\PGL}{PGL}

\def \f{{\mathbf f}}
\newcommand{\kk}{\mathbf{k}}

\newtheorem{Theorem}{Theorem}[section]
\newtheorem{Lemma}[Theorem]{Lemma}
\newtheorem{Proposition}[Theorem]{Proposition}
\newtheorem{Corollary}[Theorem]{Corollary}
\newtheorem{Definition}[Theorem]{Definition}
\newtheorem{Example}[Theorem]{Example}
\newtheorem{Remark}[Theorem]{Remark}

\newtheorem{Conjecture}[Theorem]{Conjecture}

\newtheorem{Question}[Theorem]{Question}

\newtheorem{prop}[Theorem]{Proposition}
\newtheorem{lem}[Theorem]{Lemma}
\newtheorem{example}[Theorem]{Example}
\newtheorem{remark}[Theorem]{Remark}

\newcommand{\Gm}{\mathbb{G}_m}

\newcommand{\Frob}{\operatorname{Frob}}

\newcommand{\Cartier}{\mathscr{C}}
\newcommand{\id}{\operatorname{id}}

\begin{document}

\title{Recovering affine curves over finite fields from $L$-functions}

\author{Jeremy Booher}
\address{School of Mathematics and Statistics, University of Canterbury, Private Bag 4800, Christchurch 8140, New Zealand}
\email{jeremy.booher@canterbury.ac.nz}
\urladdr{http://www.math.canterbury.ac.nz/\~{}j.booher}

\author{Jos\'e Felipe Voloch}
\address{School of Mathematics and Statistics, University of Canterbury, Private Bag 4800, Christchurch 8140, New Zealand}
\email{felipe.voloch@canterbury.ac.nz}
\urladdr{http://www.math.canterbury.ac.nz/\~{}f.voloch}

\begin{abstract}
Let $C$ be an algebraic curve over a finite field of odd characteristic.
We investigate using $L$-functions of Galois extensions of the function field $K$ of $C$ to effectively recover the curve $C$.  When $C$ is the projective line with four rational points removed, we show how to use $L$-functions of a ray class field of $K$ to effectively recover the removed points up to automorphisms of the projective line.  When $C$ is a plane curve, we show how to effectively recover the equation of $C$ using $L$-functions of Artin-Schreier covers.
\end{abstract}

\maketitle
%

%%%%%%%%%%%%%%%%%%%
\section{Introduction}%%%%%%%%
%%%%%%%%%%%%%%%%%%%

Let $K$ be a global field.
The zeta function of $K$ is an important invariant of $K$, encoding information about the places of $K$.  
By classical work of Gassmann, the zeta function of $K$ alone does not determine $K$; for example, Gassmann produces non-isomorphic number fields with the same zeta function \cite{gassmann}.  Similarly, Tate has shown that two function fields have the same zeta function if and only if the Jacobians of the corresponding curves are isogenous \cite{tate}.  It is natural to ask if additional information of a similar nature helps.

\begin{Question} \label{question1}
Do zeta functions of Galois extensions of $K$ (or $L$-functions for particular characters of the Galois group) suffice to determine $K$ up to isomorphism?  What kind of extensions are needed?  Can this be done explicitly?
\end{Question}

There has been a number of recent papers addressing instances of this question.
\begin{itemize}
\item  For example, \cite{CDL+} proves that a global field can be recovered
from its abelian $L$-functions. This requires a way to identify characters of two fields; more precisely, they prove two global fields are isomorphic if there is an identification of the groups of characters of their Abelianized absolute Galois groups preserving $L$-functions.

\item  We have shown that function fields of curves 
of genus at least two can be recovered from $L$-functions of unramified characters \cite{BV}.

\item Another approach is to attempt to recover extensions of a fixed global field from a few well-chosen $L$-functions.  For number fields, see \cites{CDL+,LiRu,Pi} and the references therein, while \cite{solomatin} considers extensions of a fixed function field.  A key difference is that all number fields are natural extensions of $\Q$, while function fields can be given as extensions of the rational function field in different ways and, in general, there is no canonical choice.
\end{itemize}
These results can all be seen as attempts to use limited information about the absolute Galois group of a global field to recover it.  (Uchida and Neukirch have shown that the absolute Galois group determines a global field \cites{neukirch,uchida76,uchida77}.)   

We will focus on $L$-functions of one-dimensional representations; for the purposes of this paper 
a character of a group $G$ is a homomorphism $G \to \mathbb{C}^\times$.
Our goal is to study Question~\ref{question1} for function fields of 
curves over finite fields with the goal of explicitly describing an effective,
finite set of $L$-functions of characters that characterizes the curve.
Our perspective is a bit different from the aforementioned results.
We work with a family of affine curves 
described in some concrete way and we recover the description using $L$-functions for uniformly-constructed covers. 
In other words, we aim to describe a set of $L$-functions that will determine
the member of the family up to some natural equivalence.

\begin{Example} \label{ex:intro}
As we will show later while proving Theorem~\ref{thm:equatoinintro}, we can recover the parameter $\lambda$ in
the Legendre family $y^2=x(x-1)(x - \lambda)$ over $\F_q$ using 
the $L$-functions of the extensions $z^p-z=\alpha x(1-y^{q-1})$ with $\alpha$ running through a basis of $\F_q/\F_p$.  See Example \ref{ex:legendre}.
\end{Example}

%We work over the finite field $\F_q$, where $q$ is a power of an odd prime $p$. 

We first tackle the special case of the projective line minus four points in odd characteristic.  In this case, the function field is known and the goal is to determine the removed points using $L$-functions up to automorphisms of the projective line. We look at $L$-functions of a suitable class field (a ray class field for the function field giving an unramified cover of the projective line minus four points) and answer an analogue of a conjecture from
\cite{SV} in a strong form. Namely, we prove:

\begin{Theorem} \label{thm:p1intro}
Let $q$ be an odd prime power and $U$ be $\PP^1_{\F_q}$ minus four $\F_q$-rational points.  Then $U$ can be recovered up to Frobenius twist from the $L$-functions 
associated to the maximal abelian extension of $\F_q(x)$ unramified over $U$ and tamely ramified on its complement. 
\end{Theorem} 

We prove two versions of this theorem.  Theorem~\ref{thm:p1distinguished} (and Remark~\ref{remark:q9} for an exceptional case) shows this can be accomplished using the $L$-function of a single well-chosen character.   
Furthermore, we can identify which character to use given only the $L$-functions for every character of the extension.
Although we need to know $U$ to compute the $L$-function, the point is that the resulting
value of the $L$-function suffices to determine $U$ uniquely up to Frobenius 
twist.
Along slightly different lines,
Theorem~\ref{thm:p1zilber} shows that if $U$ and $U'$ are $\PP^1_{\F_q}$ minus four $\F_q$-rational points and there is an isomorphism between the Galois groups of appropriate ray class fields which identifies $L$-functions then $U$ and $U'$ are Frobenius twists.  
The proofs rely on obtaining information about the reductions of these $L$-functions modulo primes above $p$ using the Cartier operator, making use of a result of R\"{u}ck \cite{ruck}.

\begin{remark}
\begin{enumerate}
\item  The first version of the theorem fits into the approach taken in \cite{LiRu} and \cite{solomatin}, while the second fits into the approach taken in \cites{BV} and \cite{CDL+} as it involves identifying characters. 

\item   
Many of our arguments extend without change to
characteristic $2$ but a few do not. In some places, we will comment on
the changes that would need to be made to extend the results but we do not
systematically address the issue.

\item Our proof in \cite{BV} used a conjecture of Bogomolov, Korotiaev and Tschinkel about Jacobians of curves of genus at least two which was proven by Zilber \cites{BKT,Z1,Z2}.  In Section~\ref{ss:backgroundl}, we extend this conjecture to generalized Jacobians of dimension at least two, and can reverse our argument and use Theorem~\ref{thm:p1zilber} to establish a special case of this extended conjecture.
 \end{enumerate}
\end{remark}

Then we tackle the general question of recovering the equation of a curve using $L$-functions of Artin-Schreier extensions.
\begin{Theorem}  \label{thm:equatoinintro}
For fixed odd prime power $q=p^r$ and $d\ge 1$, there is an explicit finite
set of polynomials $f \in \F_q[x,y]$ such that we can
recover the coefficients of any absolutely irreducible $F \in \F_q[x,y]$ of degree $d$, defining a function field $K$,
from the $L$-functions of the Artin-Schreier extensions of $K$ given by $z^p-z=f$.
\end{Theorem}

This is Theorem~\ref{thm:equations}.  It uses a very large number of $L$-functions, but can be simplified given additional information about the form of $F$ as illustrated in Example~\ref{ex:intro}.

\begin{remark}
There are many irreducible polynomials $F \in \F_q[x,y]$ of degree $d$ giving plane curves with the same function field $K$.
It is again worth emphasizing that in Theorem~\ref{thm:equatoinintro} we use the explicit description of $K$ as the function field of the curve given by $F$ in order to describe Artin-Schreier extensions of $K$ and then recover the description of $K$ by $F$ using $L$-functions of these extensions.  This is why we recover a specific $F$ and not just the function field $K$ up to isomorphism.

For example, when $-1$ is a square in $\F_q$ the fibers of the Legendre family over $\lambda$ and $1 -\lambda$ are isomorphic (via $(x,y) \mapsto (1-x,\sqrt{-1} y)$), but Example~\ref{ex:intro} is able to recover $\lambda$ because the Artin-Schreier extensions used depended on the  fiber.  Identifying the function fields of the two fibers, we end up looking at a different set of $L$-functions!
\end{remark}

\subsection*{Acknowledgements}
The authors were supported by the Marsden Fund Council administered by the Royal Society of New Zealand.  We are grateful for the assistance of Z. Brady for his help with Section~\ref{ss:technicalproof}, in particular the proofs of Proposition~\ref{prop:charsdiffer} and Lemma~\ref{lem:witness} and to the referee for the careful reading of the manuscript.

\section{\texorpdfstring{$L$}{L}-functions and Generalized Jacobians}

\subsection{Class Field Theory}

We describe the well-known connection between the class-field theoretic notions of Hilbert class field and ray class field with the language of generalized Jacobians in the function field setting.  Since we wish our extensions to have the same constant field (and have finite degree), we use the version of Hilbert and ray class fields introduced in \cites{rosen,hk}.  

Let $K/\F_q$ be a function field,  $\mathcal{S}_K$ be the set of places of $K$ and $S \subset \mathcal{S}_K$ a finite subset. Let $I_K$ denote the ideles of $K$ and $I_{K,S}$ the subgroup
of the {ideles $(\alpha_v)_{v \in \mathcal{S}_K}$ such that $\alpha_v=1$ for $v \in S$}.  
We let $\fm$ be a modulus (a formal sum of places of $K$ or, equivalently, an effective divisor) disjoint from $S$.   We allow $\fm=(0)$, the empty divisor.  {Define $I_{K,S}^{\fm}$ to be the elements of $I_{K,S}$ which are congruent to $1$ modulo $\fm$.}
The ray class field of $K$ with modulus $\fm$ is defined to be the field corresponding to the subgroup $K^\times I_{K,S}^{\fm} / K^\times$ of $I_K / K^\times$ under the class field theory correspondence.  We denote this field by $K_{S}^\fm$.

\begin{Lemma} \label{lem:maximality}
The maximal Abelian extension of $K$  in which the places of $S$ split completely and with conductor dividing $\fm$ is $K_S^\fm$.
\end{Lemma}

\begin{proof}
This is an extension of \cite[Theorem (iv)]{hk}.  Let $L$ be an Abelian extension of $K$ in which the places of $S$ split completely and with conductor dividing $\fm$.  Under the class field theory correspondence, it corresponds to an open subgroup $W \subset I_K$ containing $K^\times$.  For a place $v \in S$, we know that $K_v^\times \subset W$ as $v$ splits completely in $L$.  Since the conductor of $L$ divides $\fm$, we know that $$I^\fm_{K,S} \subset W.$$
This shows that $K^\times I_{K,S}^\fm \subset W$, and hence $L \subset K_{S}^{\fm}$.
\end{proof}

Now let $C$ be a smooth projective curve over $\F_q$ with function field $K$.  We fix a rational point $M$ of $C$, and take $S = \{v_M\}$ where $v_M$ is the place associated to $M$.  Note that the splitting condition implies the constant fields {of $K$ and $K_S^\fm$ are the same}.

We consider an abelian extension of $K$ (abelian covers of $C$) such that $v_M$ splits completely (the cover is defined over $\F_q$ and the fiber above $M$ is rational) subject to additional ramification conditions.

Let $J_\fm$ denote the generalized Jacobian of $C$ with respect to the modulus $\fm$ \cite{Serre}.
Let ${\rm supp}(\fm)$ denote the support of $\fm$, i.e, the set of places of $K$ occurring with nonzero multiplicity in $\fm$.
The Abel-Jacobi map $C \setminus {\rm supp}(\fm) \to J_\fm$ associated with a
divisor $D_1$ of degree one is the map $P \mapsto [P - D_1]$. Let $\Phi$ denote the Frobenius endomorphism of $J_\fm$.

\begin{prop} \label{prop:cftgj}
The ray class field of $K$ with modulus $\fm$ is the function field of the cover
$C'$ of $C$ given by the fiber product of the isogeny $\Phi^* - \id : J_\fm \to J_\fm$ with the Abel-Jacobi map associated to $M$.
This induces an isomorphism $\Gal(K_{S}^{\fm}/K) \simeq J_\fm(\F_q) $ and, viewing $P \in \mathcal{S}_K$ as
a divisor of degree $\deg(P)$,
the Frobenius $\Frob_P$ is identified with $[P - \deg(P) D_1] \in J_\fm(\F_q)$.
\end{prop}

\begin{proof}
Let $L$ be the function field of $C'$.  Observe that $v_M$ splits completely in $L$ as the fiber of the map $\Phi^* - 1$ over $0$ is $J_\fm(\F_q)$.  We see the cover is Galois, with Abelian Galois group $J_\fm(\F_q)$.
Furthermore, the conductor of $L$ divides $\fm$ because the cover $C' \to C \setminus {\rm supp}(\fm)$ is the pullback of an isogeny $J' \to J_\fm$.  %Serre, VI.12
  Then Lemma~\ref{lem:maximality} shows that
$L \subset K_{S}^{\fm}$, so it suffices to show that $J_\fm(\F_q) \simeq \Gal(K_{S}^{\fm}/K)$.

But class field theory in the form of \cite[Lemma 2 and Theorem]{hk} shows that $\Gal(K_{S}^{\fm}/K)$ is isomorphic to the idele $S$-ray class group modulo $\fm$.  There is a natural surjection
 from the set of divisors on $C$ with support away from $\fm$ to the set of divisors of degree $0$ on $C$ whose support is disjoint from $\fm$ given by sending a divisor $D$ to $D - \deg(D) M$.  
As in \cite[V \S 2]{Serre}, let $C_\fm^0$ denote the set of degree zero divisors on $C$ with support disjoint from $\fm$ up to $\fm$-equivalence; it is a quotient of the group of divisors of degree $0$ on $C$ whose support is disjoint from $\fm$ by the divisors of functions on $C$ which are congruent to one modulo $\fm$.  We know that $C_\fm^0 = J_\fm(\F_q)$, and
every equivalence classes of divisors on $C$ with support disjoint from $\fm$ can be represented by a divisor whose support is disjoint from $S$ as well.  But the  idele $S$-ray class group modulo $\fm$ is precisely divisors whose support is disjoint from $S$ and $\fm$ modulo the divisors of functions which are congruent to one modulo $\fm$ and are invertible at the points of $S$.  This completes the proof.
\end{proof}

\subsection{\texorpdfstring{$L$}{L}-functions} \label{ss:backgroundl}  By a \emph{pointed curve} $(C,D_1)$, we mean a smooth projective curve $C$ over $\F_q$ together with a choice of degree $1$ divisor $D_1$.  Such a divisor always exists \cite{Schmidt1931}.  
To connect with class field theory, we assume the existence of a rational point $M$ and take $D_1 = [M]$.  

  Let $\chi$ be a character of $\Gal(K_{S}^{\fm}/K) \simeq J_\fm(\F_q)$ with conductor $\fm$; this means $\fm$ is the smallest modulus such that $\chi$ factors through $\Gal(K_{S}^{\fm}/K)$.  It is a character of the absolute Galois group of $K$ that is
trivial on the absolute Galois group of $\F_q$.  

\begin{Definition}
For a place $P$ of $K$ not in 
${\rm supp}(\fm)$,
let $\Frob_P \in \Gal(K_{S}^{\fm}/K)$ denote the Frobenius at $P$. 
For any character $\chi$ with conductor $\fm$ and ${\rm supp}(\fm)$ disjoint from $S$, the $L$-series (or $L$-function) for $\chi$ is defined as
\[
L(T,C,\chi) := \prod_{P \in \mathcal{S}_C\setminus {\rm supp}(\fm)} ( 1 - \chi(\Frob_P) T^{\deg P})^{-1}.
\]
\end{Definition}

We sometimes use the notation $L(T,D/C,\chi)$ in the case of a Galois cover
$D/C$ if we can view $\chi$ as a character of $\Gal(D/C)$.

Note that this definition depends on the choice of marked point $M$ as the Abel-Jacobi map determined by $M$ is used to identify the Galois group with $J_\fm(\F_q)$.
It is a standard calculation that 

\begin{equation}
\label{eq:standard}
\frac{L'(T,C,\chi)}{L(T,C,\chi)} =
\sum_{P \in \mathcal{S}_C} 
\frac{(\deg P) \chi(\Frob_P) T^{\deg P -1}}{( 1 - \chi(\Frob_P) T^{\deg P})} =
\frac{d}{dT}\left(\sum_{n=1}^{\infty} S_n(\chi) \frac{T^n}{n}\right)
\end{equation}
where 
$\displaystyle S_n(\chi) = \sum_{P \in U(\F_{q^n})} \chi(P+\Phi(P)+\cdots+\Phi^{n-1}(P))$ and $U = C \setminus {\rm supp}(\fm)$.

One of the main purposes of this paper is to study the following
conjecture:

\begin{Conjecture} \label{conj:lfun}
Let $(C,[M])$ and $(C',[M'])$ be smooth irreducible projective pointed curves 
over a finite field~$\F_q$.  Let $\fm,\fm'$ be moduli on $C,C'$ disjoint from $M,$ $M'$ respectively.  Suppose the corresponding generalized
Jacobians $J_{\fm},J_{\fm'}'$ have dimension at least two, and that there is a set-theoretic map $\psi : J_{\fm'}'(\overline{\F}_q) \to J_{\fm}(\overline{\F}_q)$ inducing an isomorphism of groups between $J_{\fm}(\F_{q^n})$ and $J_{\fm'}'(\F_{q^n})$ 
for every $n \ge 1$.
If $L(T, C \otimes \F_{q^n}, \chi) = L(T, C' \otimes \F_{q^n}, \chi \circ \psi|_{
J_{\fm'}'(\F_{q^n})})$ for all $n$ and all characters $\chi$ of $J_{\fm}(\F_{q^n})$, then 
$C$ and $C'$ are Frobenius twists of each other and the map
$\psi$ arises from a morphism of curves composed with a limit of Frobenius maps
which sends $M$ to $M'$ and $\fm$ to $\fm'$.
\end{Conjecture}

\begin{remark}
More precisely, the last condition means that 
there exists an integer $m$, an isomorphism $\alpha: \Frob^m(C) \simeq C'$, and a generalized Frobenius $\beta: \overline{\F}_q \to \overline{\F}_q$ restricting to $\Frob^{-m}$ on $\F_q$ such that $\alpha \circ \beta : C \to C'$ sends $M$ to $M'$, $\fm$ to $\fm'$, and such that $\alpha \circ \beta$ induces $\psi$ on generalized Jacobians.  
\end{remark}

Conjecture \ref{conj:lfun} was proved in \cite{BV} for $\fm = (0)$ as a consequence of a result conjectured 
by Bogomolov, Korotiaev and Tschinkel \cite{BKT} and proved by
Zilber \cites{Z1,Z2}.  That result is the $\fm = (0)$ case of the following
conjecture.

\begin{Conjecture} \label{conj:zilber}
Let $(C,[M])$ and $(C',[M'])$ be smooth, projective, irreducible pointed curves over $\F_q$
and $\fm,\fm'$ moduli on $C,C'$ respectively.  Suppose the generalized
Jacobians $J_{\fm}$ and $J_{\fm'}'$ have dimension at least two, and let $U , U'$ be the complements of the support of $\fm,\fm'$ in $C,C'$ respectively.  If  $\psi : J_{\fm}(\overline{\F}_q) \to J_{\fm'}'(\overline{\F}_q)$ is a group isomorphism such that $\psi(U(\overline{\F}_q)) = U'(\overline{\F}_q)$ (with respect to the Abel-Jacobi embeddings determined by $M$ and $M'$), then $\psi$ arises from a morphism of curves composed with a limit of Frobenius maps which sends $\fm$ to $\fm'$.

 More precisely, there exists an integer $m$, an isomorphism $\alpha: \Frob^m(C) \simeq C'$, and a generalized Frobenius $\beta: \overline{\F}_q \to \overline{\F}_q$ restricting to $\Frob^{-m}$ on $\F_q$ such that 
$\alpha \circ \beta : C \to C'$ sends $\fm$ to $\fm'$ and $M$ to $M'$, and that
induces $\psi$ on generalized Jacobians.
\end{Conjecture}

The following result relates the two above conjectures.

\begin{Theorem} \label{thm:equivalence}
Conjecture \ref{conj:zilber} implies Conjecture \ref{conj:lfun}. Conversely,
assuming Conjecture \ref{conj:lfun}, 
the hypotheses of Conjecture \ref{conj:zilber} and, in addition, that
a power of Frobenius commutes with $\psi$, then the conclusion of 
Conjecture \ref{conj:zilber} holds.
\end{Theorem}

\begin{proof}
Assume Conjecture \ref{conj:lfun} and, in addition, that the $m$-th
power of Frobenius commutes with $\psi$.
It follows directly that the exponential sums $S_n(\chi)$ for $m|n$ in \eqref{eq:standard} match for each character.
Hence \eqref{eq:standard} gives equality of
$L$-functions over extensions of $\F_{q^m}$ and from that we get the morphism. 

The proof in the opposite direction is almost verbatim the proof of
\cite[Theorem 2.5]{BV} replacing linear equivalence of divisors with
the stricter equivalence modulo the fixed modulus. The argument there works without 
change unless the curves have genus zero. In this latter case, points
are linearly equivalent to each other so we need to show that they are not equivalent
modulo a fixed modulus under the hypothesis of Conjecture \ref{conj:zilber}. Such a modulus has degree at least two. 

If two points in $\PP^1$ given by $a, b$ are equivalent modulo $\fm$ and
we assume (without loss of generality) that $\infty$ is in the support of $\fm$, then we conclude that
$(x-a)/(x-b) \equiv 1 \pmod{\fm}$. We distinguish two cases. First assume that there is
a point $c$ other than $\infty$ in the support of $\fm$. As
$(x-a)/(x-b) \equiv 1 \pmod{\fm}$ we see $(c-a)/(c-b) = 1$ and hence $a=b$.
If $\fm$ is supported only at infinity, our hypothesis requires 
$(x-a)/(x-b) \equiv 1 \pmod{2\infty}$. But
$$(x-a)/(x-b) = (1-a/x)/(1-b/x) = 1 +(b-a)/x +O(1/x^2)$$
and hence $(x-a)/(x-b) \equiv 1 \pmod{2\infty}$ only if $a=b$.
\end{proof}

\begin{Remark}
In the next section, we will prove a strong form of Conjecture \ref{conj:lfun}
when $C = \PP^1$ and $\fm$ is a sum of four distinct places of degree one. 
Under these hypotheses
the generalized Jacobians are tori \cite[VI \S 3]{Serre}.
For tori, the $\F_q$-Frobenius is multiplication by $q$ in the group law so it commutes with $\psi$
automatically. Theorem \ref{thm:equivalence} will show that Conjecture \ref{conj:zilber} holds in this case.

In \cite[Theorem 5.11]{BKT} it was proved that a power of Frobenius commutes with $\psi$ in the case
$\fm =(0)$ before conjecture \ref{conj:zilber} was proved for $\fm=(0)$. It would be interesting to see if the
argument there can be generalized.

If the generalized Jacobians have dimension one, the condition on the image
of the Abel-Jacobi embedding in Conjecture \ref{conj:zilber} is clearly vacuous and
the condition on the $L$-functions in Conjecture \ref{conj:lfun} can also
be shown to be vacuous. In the special case that the curves have genus one
and both moduli are $(0)$, \cite[Remark 2.7]{BV} provides examples where 
the conclusion of both conjectures fail. The other possibility is that
the curves have genus zero and the moduli have degree two. In this case,
the curves are automatically isomorphic under the conjectures' hypotheses
but it is possible to construct
group isomorphisms $\psi$ that do not arise from morphisms of curves.
\end{Remark}

%%%%%%%%%%%%%%%%%%%
\section{The Projective Line}%%%%%%%% 
\label{sec:p1}
%%%%%%%%%%%%%%%%%%%

Let $q=p^r$ be an odd prime power.  In this section, we will show how to recover $\PP^1_{\F_q}$ with four marked rational points from $L$-functions of covers unramified away from the marked points.

\subsection{Covers of the Projective Line}

For $\lambda \in \F_q \backslash \{0,1\}$, consider the modulus  $$\fm(\lambda) = [0] + [1] + [\lambda]$$ on $\PP^1_{\F_q}$  and let $U_\lambda \colonequals \PP^1_{\F_q} \backslash \{0,1, \lambda\}$.  
\begin{Definition}
Let $C_\lambda$ be the projective curve over $\F_q$ given by the projective equation
\begin{equation} \label{eq:clambda}
 \lambda x^{q-1} - y^{q-1} + (1-\lambda)z^{q-1} =0.
 \end{equation}
Let $\pi : C_\lambda \to \PP^1_{\F_q}$ be the map sending $[x,y,z]$ to $[z^{q-1},z^{q-1}-x^{q-1}]$.  
\end{Definition}

A straightforward calculation checks that $C_\lambda$ is smooth.  Note there is a natural action of $\mu_{q-1}^3$ modulo the diagonal $\mu_{q-1}$, acting by roots of unity on the three coordinates.  

\begin{Proposition} \label{prop:rcf}
The cover $\pi: C_\lambda \to \PP^1_{\F_q}$ is the ray class field cover of $\PP^1_{\F_q}$ with respect to $\fm(\lambda)$ in which the infinite place splits completely.  In particular, 
\begin{enumerate}
\item \label{rcf1}  $\pi$ is a degree $(q-1)^2$ branched Galois cover with Galois group $\mu_{q-1}^2$;
\item \label{rcf2} there are $q-1$ points of $C_\lambda$ with ramification index $q-1$ above each point of $\fm(\lambda)$, and the infinite place splits completely;
\item \label{rcf3} over $U_\lambda$, $\pi$ is the pullback of $\Phi^* - 1 : J_{\fm(\lambda)} \to J_{\fm(\lambda)}$ by the Abel-Jacobi map $U_\lambda \to J_{\fm(\lambda)}$ which sends $t$ to $[t] - [\infty]$ (as before $\Phi$ is the Frobenius map on $J_\fm$).
\end{enumerate}
\end{Proposition}

\begin{proof}
We know that $J_{\fm(\lambda)}$ is isomorphic to $\Gm^3$ modulo the diagonal $\Gm$ \cite[V \S3.13 Prop 7]{Serre};
hence $J_{\fm(\lambda)} \simeq \Gm^2$ by choosing the first coordinate to be $1$.  The Abel-Jacobi map sends $t$ to the divisor $[t]-[\infty]$ which is the divisor of the function $f(x)=x-t$.  Thus the Abel-Jacobi map is given by sending $t$ to $[(-t,1-t,\lambda-t)]$ in $\Gm^3/\Gm$, or equivalently sending $t$ to $(1-1/t,1-\lambda/t)$ in $J_{\fm(\lambda)} \simeq \Gm^2$.
Letting $x$ and $y$ be the coordinates of $J_{\fm(\lambda)}$, as the Frobenius on $\Gm$ is the $q$th power map pulling back $\Phi^* - 1$ gives the affine curve 
\[
x^{q-1} = (1-1/t), \quad y^{q-1} = (1 - \lambda/t).
\]
Eliminating $t$ and writing it as a projective equation gives \eqref{eq:clambda} and the map $\pi$.  We directly see the claim about ramification.  Note that $\infty$ splits completely over $\F_q$ as the Abel-Jacobi map sends $\infty$ to $(1,1) \in \Gm^2$ and the fiber over $(1,1)$ is $\mu_{q-1} \times \mu_{q-1}$.

This shows that $C_\lambda$ is the ray class field cover of $\PP^1_{\F_q}$ with respect to $\fm(\lambda)$ in which the infinite place splits completely.  The Galois group is identified with $J_{\fm(\lambda)}(\F_q) \simeq \mu_{q-1}^2$ by having the group act via multiplication by roots of unity on $x$ and $y$.
\end{proof}

\begin{Corollary}
The genus of $C_\lambda$ is $\frac{1}{2} (q-2)(q-3)$.
\end{Corollary}

\begin{proof}
This follows from the Riemann-Hurwitz formula, or the fact that $C_\lambda$ is a smooth plane curve of degree $q-1$.
\end{proof}

We now record some additional information about $C_\lambda$ for use later.
Let $G \colonequals \mu_{q-1}^2 \simeq J_{\fm(\lambda)}(\F_q)$ be the Galois group of $\pi$, acting on $x$ and $y$ via roots of unity.  Let $G_1$ (resp. $G_\lambda$) be copies of $\mu_{q-1}$ embedded into $G$ as the first (resp. second) coordinates, and $G_0$ be a copy of $\mu_{q-1}$ embedded diagonally into $G$. 

\begin{Lemma} \label{lem:inertia}
The inertia groups at the $q-1$ points of $C_\lambda$ over $0$ (resp. $1$, $\lambda$) are $G_0$ (resp. $G_1$, $G_\lambda$).
\end{Lemma}

\begin{proof}
Computing using $\pi$, we see the points of $C_\lambda$ over $0$ are those with $z=0$, and hence are fixed by $G_0$.  The other branch points are similar.
\end{proof}

\begin{Lemma} \label{lem:differentials}
A basis for the space of regular differentials on the projective curve given by an affine equation $a x^{q-1} + b y^{q-1} + 1=0$ over $\F_q$ is given by $\displaystyle \omega_{i,j} \colonequals \frac{x^i y^j}{y^{q-2}} dx$ where $i,j \geq 0$ and $i+j \leq q-4$.  
\end{Lemma}

\begin{proof}
This is standard.
\end{proof}

\subsection{L-Functions and the Cartier Operator} \label{ss:lfunctions}
Recall that the group $G = \mu_{q-1}^2$ is the Galois group of the cover $\pi : C_\lambda \to \PP^1_{\F_q}$.
Given a pair $(b,c) \in \Z/(q-1)$, we define a character $\chi_{b,c} : \mu_{q-1}^2 \to \mu_{q-1}$ by
\[
\chi_{b,c}(\zeta_1,\zeta_2) = \zeta_1 ^ b \zeta_2^c .
\]
%This realizes the (non-canonical) isomorphism of $G$ with its character group.
Any character of $G$ equals $\chi_{b,c}$ for appropriately chosen $(b,c) \in \Z/(q-1)^2$.  Let $H_{b,c}$ denote the kernel of $\chi_{b,c}$.

We will use the Cartier operator to obtain information about the $L$-functions for the cover $\pi : C_\lambda \to \PP^1_{\F_q}$.  For our purposes, the Cartier operator may be viewed as a (semi-linear) operator $\Cartier$ on the space of regular differentials $H^0(C_\lambda, \Omega^1_{C_\lambda})$.  If $q = p^r$, the $r$th power of the Cartier operator on $C_\lambda$ is $\F_q$-linear and may be computed explicitly using the following result.  We define $\alpha_{m,n}$ to be the multinomial coefficient
\begin{equation} \label{eq:defalpha}
\alpha_{m,n} \colonequals \binom{q-1}{m,n,q-1-m-n} = \frac{(q-1)!}{m!\,n!\,(q-1-m-n)!}.
\end{equation}

\begin{Proposition} \label{prop:cartieraction}
Let $X$ be the projective curve over $\F_q$ determined by the equation $f(x,y)=ax^{q-1}+by^{q-1}+1=0$ with $ab \neq 0$.  Writing $q = p^r$, the $r$-th power of the Cartier operator acts diagonalizably on $H^0(X,\Omega^1_X)$ with eigenvalues
\[
\alpha_{i+1,j+1} a^{i+1} b^{j+1} \quad \text{for } 0 \leq i,j \text{ with }  i+j \leq q-4.
\]
The corresponding eigenvectors are $\displaystyle \omega_{i,j} = \frac{x^i y^j }{y^{q-2}} dx$.
\end{Proposition}

Note that the eigenvalues are non-zero if $q=p$ but can be zero when
$q \ne p$ as $\alpha_{i+1,j+1}$ may be zero modulo $p$.

\begin{proof}
By \cite[Theorem 1.1 and proof of Theorem 4.1]{StV} the $r$-th power of the Cartier operator on a regular differential $h f_y^{-1} dx$ (with $h \in \F_q[x,y]$)
is given by 
$$
\mathscr{C}^{r}\left(h f_{y}^{-1} d x\right)=\left(\left(\frac{\partial}{\partial x}\right)^{(q-1)}\left(\frac{\partial}{\partial y}\right)^{(q-1)}\left(f^{q-1} h\right)\right)^{1/q} f_{y}^{-1} d x
$$
where $\left(\frac{\partial}{\partial x}\right)^{(q-1)}$ and $\left(\frac{\partial}{\partial y}\right)^{(q-1)}$ are partial Hasse derivatives.  
The space of regular differentials of $X$ is 
$\left\{h f_{y}^{-1} d x | h \in \F_q[x, y] , \deg h \le q-4\right\}$ as in Lemma \ref{lem:differentials}.

We will see that the action of $\mathscr{C}^{r}$ on the basis obtained by taking $h=x^i y^j$ is diagonal. 
To see this, notice that $\left(\frac{\partial}{\partial y}\right)^{(q-1)} ( x^i y^j) =0$ unless $j \equiv q-1 \pmod{q}$, with a similar statement for the partial Hasse derivative with respect to $x$. The only term of $f^{q-1} x^i y^j = (a x^{q-1} + b y^{q-1} +1)^{q-1} x^i y^j$ which contributes 
is the $x^{(i+1)(q-1)} y^{(j+1) (q-1)}$ term of $ f^{q-1}$.  Thus we see that
\begin{align*}
\Cartier^{r} \left( x^i y^j f_y^{-1} dx  \right) &=  \left( \left(\frac{\partial}{\partial x}\right)^{(q-1)}\left(\frac{\partial}{\partial y}\right)^{(q-1)} (\alpha_{i+1,j+1} a^{i+1} b^{j+1} x^{(i+1)(q-1) +i} y^{(j+1)(q-1)+j} ) \right)^{1/q} f_y^{-1} dx \\
&= \alpha_{i+1,j+1} a^{i+1} b^{j+1} x^i y^j f_y^{-1} dx. \qedhere
\end{align*}
\end{proof}

The Cartier operator gives significant arithmetic information about smooth projective curves over $\F_q$, in particular the curve $C_\lambda$.

\begin{Definition}
The $p$-rank of $C_\lambda$, denoted $f_{C_\lambda}$, is the $\F_q$-dimension of $H^0(C_\lambda,\Omega^1_{C_\lambda})^{\textrm{ss}}$, the subspace of $H^0(C_\lambda,\Omega^1_{C_\lambda})$ where $\Cartier$ is invertible.  

More generally, for a character $\chi : G \to \mu_{q-1}$ we define $f_{C_\lambda,\chi} := \dim_{\F_q} H^0(C_\lambda,\Omega^1_{C_\lambda})^{\chi,\textrm{ss}}$ to be the dimension of the subspace of regular differentials where $\Cartier$ acts semisimply and $\Gal(C_\lambda/\PP^1_{\F_q} ) =G$ acts via $\chi$.  

We also  let $P_\lambda(T,\chi)$ denote the numerator of the Artin $L$-function $L(T,C_\lambda/\PP^1_{\F_q},\chi)$.
\end{Definition}

Note that $P_\lambda(T,\chi) = L(T,C_{\lambda}/\PP^1_{\F_q},\chi)$ unless $\chi$ is trivial, in which case $L(T,C_\lambda/\PP^1_{\F_q},1)$ has a denominator.  

\begin{Proposition} \label{prop:ruck}
Letting $q=p^r$, we have 
\begin{equation} \label{eq:ruck}
P_\lambda(T,\chi) \equiv \det( 1 - \mathscr{C}^r T | H^0(C_\lambda, \Omega^1_{C_\lambda})^{\chi}) \pmod{p}.
\end{equation}
In particular, the degree of $P_\lambda(T,\chi)$ is $f_{C_\lambda,\chi}$. 
\end{Proposition}

\begin{proof}
This is a special case of a result of R\"{u}ck \cite[Theorem 4.1 and Corollary 4.1]{ruck}.  Note that $P_\lambda(T,\chi)$ can be viewed as a polynomial with coefficients in the ring of Witt vectors $W(\F_q)$ as this ring contains the $(q-1)$-st roots of unity, so reducing modulo $p$ makes sense.  
The identification of $\mu_{q-1} \subset W(\F_q)$ with the complex $(q-1)$-st
roots of unity (used to define the $L$-function) requires an arbitrary choice; we fix one throughout.
\end{proof}

\begin{Remark}
The result of R\"{u}ck generalizes a result of Manin \cite{manin} which relates the numerator of the zeta function with the action of the Cartier operator on the space of regular differentials.
\end{Remark}

\begin{Proposition} \label{prop:charcomputation}
Fix a non-trivial character $\chi_{b,c}$ and let $i \colonequals (b-1 \mod{q-1})$ and $j \colonequals (c-1 \mod{q-1})$.  
If $i + j \leq q-4$ and $\alpha_{i+1,j+1} \not \equiv 0 \pmod{p}$ then $f_{C,\chi_{b,c}} =1$ and 
\[
P_\lambda(T,\chi_{b,c}) \equiv 1 - \alpha_{i+1,j+1} \left( \frac{ \lambda}{1-\lambda} \right) ^{i+1} \left( \frac{ 1}{\lambda-1} \right)^{j+1} T \pmod{p}.
\]
Otherwise $f_{C_\lambda,\chi_{b,c}} =0$ and $P_\lambda(T,\chi_{b,c}) \equiv  1 \pmod{p}$.
\end{Proposition}

\begin{proof}
Consider the affine patch of $C_\lambda$ obtained by taking $z=1$, given by 
\begin{equation}
\frac{\lambda}{1-\lambda} x^{q-1} + \frac{1}{\lambda-1} y^{q-1} + 1 =0.
\end{equation}
 A basis for the regular differentials on $C_\lambda$ in this coordinate system is given by
\[
\omega_{m,n} \colonequals \frac{x^m y^n }{y^{q-2}} dx
\]  
with $m,n \geq 0$ and $m + n \leq q-4$ by Lemma~\ref{lem:differentials}.
The Galois action of $G = \mu_{q-1}^2$ is acting by roots of unity on $x$ and $y$, so for $(\zeta_1,\zeta_2) \in \mu_{q-1}^2$ we compute that
\[
(\zeta_1,\zeta_2)^* \omega_{m,n} = \frac{(\zeta_1 x)^m (\zeta_2 y)^n }{(\zeta_2 y)^{q-2}} d(\zeta_1 x) = \zeta_1^{m+1} \zeta_2^{n-(q-2)} \omega_{m,n} = \zeta_1^{m+1} \zeta_2^{n+1} \omega_{m,n}.
\]
On the other hand, $\chi_{b,c}(\zeta_1,\zeta_2) = \zeta_1^b \zeta_2^c$.  Thus $\omega_{m,n}$ lies in $H^0(\Omega_{C_\lambda}^1)^{\chi_{b,c}}$ if and only if $m +1 \equiv b \pmod{q-1}$ and $n+1 \equiv c \pmod{q-1}$.
Thus we see that $H^0(\Omega_{C_\lambda}^1)^{\chi_{b,c}}$ is one dimensional if and only if $i + j \leq q-4$.  
(When this happens, $\omega_{i,j}$ spans the space.)

If $i +j > q-4$, the dimension of $H^0(\Omega_{C_\lambda}^1)^{\chi_{b,c}}$ is zero and hence $f_{C_\lambda,\chi_{b,c}} =0$ and $P_\lambda(t,\chi_{b,c}) =0$ by Proposition~\ref{prop:ruck}.  If $i + j \leq q-4$ then Proposition~\ref{prop:cartieraction} shows that
\[
\Cartier^r \omega_{i,j} = \alpha_{i+1,j+1} \left(\frac{\lambda}{1-\lambda} \right)^{i+1} \left(\frac{1}{\lambda-1}\right)^{j+1} \omega_{i,j}.
\]
By Proposition~\ref{prop:ruck}, we conclude that
\[
P_\lambda(T,\chi_{b,c}) \equiv 1 - \alpha_{i+1,j+1} \left(\frac{\lambda}{1-\lambda} \right)^{i+1} \left(\frac{1}{\lambda-1}\right)^{j+1}  T \pmod{p}. \qedhere
\]
\end{proof}

\begin{Corollary} \label{cor:congruence}
If $q \neq 3$, we have that $P_\lambda(T,\chi_{1,1}) \equiv 1 + 2  \lambda  (\lambda-1)^{-2} T \pmod{p}$.  (If $q=3$, $P_\lambda(T,\chi_{1,1}) = 1$.)
\end{Corollary}

\begin{proof}
When $q \neq 3$, take $i = j = 0$ in Proposition~\ref{prop:charcomputation}.  In the edge case that $q=3$, by inspection $C_\lambda$ has genus zero so $P_\lambda(T,\chi_{1,1}) = 1$.
\end{proof}

For $\lambda_1$ and $\lambda_2 \in \F_q \backslash \{0,1\}$ we have identified $G$ as the Galois group of $C_{\lambda_1} \to \PP^1_{\F_q}$ and of $C_{\lambda_2} \to \PP^1_{\F_q}$. 

\begin{Corollary} \label{cor:identifyfibers}
 If $L(T,C_{\lambda_1}/\PP^1_{\F_q},\chi_{1,1}) \equiv L(T, C_{\lambda_2}/\PP^1_{\F_q},\chi_{1,1}) \pmod{p}$ then $\lambda_2 = \lambda_1$ or $\lambda_2 = 1/\lambda_1$.  
\end{Corollary}

\begin{proof}
Let $f(\lambda) = 2 \lambda (\lambda-1)^{-2}$.  As $f(\lambda) = f(1/\lambda)$ and $f(\lambda) = c$ has at most two solutions for fixed $c \in \F_q$, this follows from Corollary~\ref{cor:congruence}.
\end{proof}

We can remove the ambiguity using a second $L$-function.

\begin{Corollary}  \label{cor:identifyfibers2}
Suppose $p \geq 5$.  
If $L(T,C_{\lambda_1}/\PP^1_{\F_q},\chi_{1,1}) \equiv L(T, C_{\lambda_2}/\PP^1_{\F_q},\chi_{1,1}) \pmod{p}$ and $L(T,C_{\lambda_1}/\PP^1_{\F_q},\chi_{1,2}) \equiv L(T, C_{\lambda_2}/\PP^1_{\F_q},\chi_{1,2}) \pmod{p}$, then $\lambda_1 = \lambda_2$.
\end{Corollary}

\begin{proof}
For $b =1, c=2$ we have $i=0$ and $j=1$, so $i+j=1 \leq q-4$ by hypothesis.   
Using Proposition~\ref{prop:charcomputation} we see that
\begin{align*}
P_\lambda(T,\chi_{1,2}) &\equiv 1 - \alpha_{1,2} \left( \frac{\lambda}{1-\lambda} \right) \left( \frac{1}{\lambda-1} \right)^2 T \pmod{p} \\
& \equiv 1 - 3 \lambda (\lambda-1)^3 T \pmod{p}
\end{align*}
as $i = q-4$ and $j=0$ and the exponents only matter modulo $q-1$.  The linear terms of the two pairs of $L$-functions being congruent modulo $p$ implies that  
\begin{equation*}
\frac{\lambda_1}{(\lambda_1-1)^2} \equiv \frac{\lambda_2}{(\lambda_2-1)^2} \pmod{p} \quad \text{and} \quad
\frac{ \lambda_1}{(\lambda_1-1)^3} \equiv  \frac{ \lambda_2}{(\lambda_2-1)^3} \pmod{p}.
\end{equation*}  
Hence we conclude that $\lambda_1 = \lambda_2$.
\end{proof}

\begin{remark}
These corollaries use the non-canonical identification of $G$ with the Galois group of the cover $C_{\lambda_i} \to \PP^1_{\F_q}$ to specify the characters.  The identification is natural given the equation \eqref{eq:clambda} (and independent of $\lambda$), but this is not intrinsic to the curve.  To distinguish $C_{\lambda_1}$ and $C_{\lambda_2}$ as covers of $\PP^1$ without this non-canonical identification, we must intrinsically identify the character $\chi_{1,1}$.  We will address this in the next subsection.
\end{remark}

\subsection{Distinguishing Subsets via \texorpdfstring{$L$}{L}-Functions} \label{ss:distinguishing}

Let $U$ be an open subset of $\PP^1_{\F_q}$ which is the complement of a finite set of $\F_q$-points $P_1, \ldots, P_n$.  If $n \leq 3$, then all such $U$ with $n$ points removed are isomorphic as $\PGL_2(\F_q)$ acts three-transitively on $\PP^1_{\F_q}$.  If $n = 4$, we wish to determine $U$ up to isomorphism (equivalently, the $4$ points up to automorphism of $\PP^1_{\F_q}$) using $L$-functions of characters of covers of $U$ (covers of $\PP^1_{\F_q}$ unramified away from $P_1, \ldots, P_4$).  We will do so using characters of the ray class field with modulus $[P_2] + [P_3] + [P_4]$ in which $P_1$ splits completely.  The key technical obstacle is finding a way to specify the character to use.

\begin{Proposition} \label{prop:distinguishedchar}
Continuing the notation of Section \ref{ss:lfunctions}, if $q \neq 9$ and the character $\chi: \Gal(C_\lambda/\PP^1_{\F_q}) = G \to \mu_{q-1}$ satisfies
\begin{enumerate}

\item \label{cond1} $\chi$ is surjective;

\item  the genus of $C_\lambda / \ker \chi$ is $(q-3)/2$;

\item  the cover $C_\lambda \to C_\lambda/ \ker \chi$ is ramified only at the $q-1$ points of $C_\lambda$ lying over $0$;

\item  \label{cond4} $f_{C_\lambda,\chi^n} = 1$ if and only if $f_{C_\lambda,\chi_{n,n}} =1$ for all $n$ such that $0< 2n < q-1$;  
\end{enumerate}
then $\chi$ is $\chi_{1,1}, \chi_{p,p}, \ldots,$ or $\chi_{p^{r-1},p^{r-1}}$.  If $q=9$, $\chi$ could additionally be $\chi_{1,3}$ or $\chi_{3,1}$.
\end{Proposition}

We can reinterpret (\ref{cond4}) using Proposition~\ref{prop:charcomputation}.
For an integer $n$, the condition $f_{C_\lambda,\chi_{n,n}} =1$ is equivalent to the condition $C(n)$ that
 \begin{equation} \label{eq:condn}
C(n) : \quad  0<2n < q-1 \quad \text{and} \quad (n \mod{p^i}) < p^i/2 \text{ for each } i \in \{1,\ldots,r-1\}.
\end{equation}
(Note that the first piece of $C(n)$ is almost equivalent to $(n \mod{p^r}) < p^r/2$: the only difference is that $n=(p^r-1)/2$ and $n=0$ are excluded.)   %satisfying $0 <  b ,c < q-1$,
For integers $b$ and $c$, write $b' = (b \mod{q-1})$ and $c' = (c \mod{q-1})$ and let $C(b,c)$ be the condition that
\begin{equation} \label{eq:condbc}
C(b,c) : \quad b' + c' < q-1 \quad \text{ and } \quad \alpha_{b,c} = \frac{(q-1) (q-2) \cdots (q-b'-c')}{b'! \cdot c'!} \not \equiv 0 \pmod{p}.
\end{equation}
Proposition~\ref{prop:charcomputation} shows that $C(nb,nc)$ holds if and only if $f_{C_\lambda, \chi_{b,c}^n}=1$. 

\begin{remark} \label{remark:kummer}
The conditions $C(n)$ and $C(b,c)$ may initially appear unconnected.  Note that $\alpha_{b,c} \equiv \pm \binom{b'+c'}{b'} \mod{p}$.  When $b' + c' < q = p^r$, a theorem of Kummer \cite{kummer} implies this is non-zero modulo $p$ precisely if $(b' \mod p^i) + (c' \mod p^i) < p^i$ for all $1 \leq i \leq r-1$.  When $0 < 2n < q-1$, we therefore directly see that $C(n)$ is equivalent to $C(n,n)$. 
\end{remark}

We will prove Proposition~\ref{prop:distinguishedchar} relying on the following result: 

\begin{Proposition} \label{prop:charsdiffer}
Let $b$ and $c$ be integers that are relatively prime to $q-1$ satisfying $0<b,c < q-1$.  For $0 < n < 2q$, suppose that $C(n)$ holds if and only if $C(nb,nc)$ holds.  Then $b = c = p^i $ for some $i \in \{0,1,\ldots, r-1\}$, or $q=9$ and $(b,c) = (3,1)$ or $(1,3)$.
\end{Proposition}

We defer the proof of Proposition~\ref{prop:charsdiffer} until Section~\ref{ss:technicalproof}.  The proof is elementary but complicated.

\begin{proof}[Proof of Proposition~\ref{prop:distinguishedchar}]
Suppose $\chi$ is a character of $G$ satisfying (\ref{cond1})-(\ref{cond4}). %, and take $H = \ker \chi$.
As $\chi$ is surjective, the degree of the quotient map $C_\lambda \to C_\lambda /\ker \chi$ is $q-1$.  As this map is Galois, the ramification index of all of the points of $C_\lambda$ above zero are the same; denote their common value by $e$.  Then the Riemann Hurwitz formula implies that
\[
(q-2)(q-3) -2 = (q-1) (q-3 -2) + (q-1)(e-1).
\]
Thus $e=2$.

Write $\chi = \chi_{b,c}$ and $H_{b,c} = \ker \chi_{b,c}$.  By considering the intersection of the inertia groups given in Lemma~\ref{lem:inertia} with $H_{b,c}$, we see that the ramification indices of $C_\lambda \to C_\lambda/H_{b,c}$ at the points of $C_\lambda$ above $0$ (resp. $1$, $\lambda$) are $\gcd(b+c,q-1)$ (resp. $\gcd(b,q-1)$, $\gcd(c,q-1)$).  Given that $C_\lambda \to C_\lambda /\ker \chi$  is ramified only at the points over $0$ with ramification index $2$, we conclude that
\[
\gcd(b+c,q-1)=2, \quad \gcd(b,q-1)=1, \quad \text{and } \gcd(c,q-1)=1.
\]
Applying Proposition~\ref{prop:charsdiffer} and the reformulation of (\ref{cond4}) in terms of $C(n)$ and $C(n,n)$, we see that $\chi = \chi_{b,c} = \chi_{1,1}^{p^i}$ for $i \in \{0,1,\ldots,r-1\}$ unless we are in the exceptional case when $q=9$.
\end{proof}

Proposition~\ref{prop:distinguishedchar} lets us compare $L$-functions without having to identify the Galois groups.  

\begin{Theorem} \label{thm:p1distinguished}
Suppose that $q = p^r$ is odd and not equal to $9$, and let $U$ and $U'$ be $\PP^1_{\F_q}$ with four removed rational points.  Fix orderings of the removed points ($P_1,\ldots,P_4$ and $P'_1,\ldots,P_4'$).  Let $C$ be the ray class field for $\PP^1_{\F_q}$ with modulus $\fm = [P_2] + [P_3] +[P_4]$ such that $P_1$ splits completely.  Fix a character $\chi : \Gal(C / \PP^1_{\F_q}) \to \mu_{q-1}$ such that
\begin{enumerate}
\item  $\chi$ is surjective;
\item  the genus of the quotient $C/\ker \chi$ is $(q-3)/2$;
\item  the cover $C \to C/\ker \chi$ is ramified only at the $q-1$ points of $C$ lying over $P_2$;
\item  $f_{C,\chi^n}=1$ if and only if $C(n)$ holds for all integers $0 < 2n < q-1$.
\end{enumerate}
Such a character $\chi$ exists, and is ``unique up to Frobenius'': the characters $\chi,\chi^p, \ldots \chi^{p^{r-1}}$ are the only characters of $\Gal(C/\PP^1_{\F_q})$ with these properties.  

Let $C'$ and $\chi'$ be defined analogously for $U'$.  If the derivatives $L'(0,\PP^1_{\F_q},\chi)$ and $L'(0,\PP^1_{\F_q},\chi')$
lie in the same Frobenius orbit of $\PP^1_{\F_q}$, in other words
there exists an integer $0\leq i <r$ such that 
\[
(L'(0, \PP^1_{\F_q},\chi) \mod{p}) = (L'(0,\PP^1_{\F_q},\chi') \mod{p}) ^{p^i},
\]
then there is an automorphism $\alpha$ of $\PP^1_{\F_q}$ such that $\alpha(U) = U'$, $\alpha(P_1)=P'_1$, and $\alpha(P_2)=P'_2$. 
\end{Theorem}

\begin{Remark}
Informally, this says there exists a character $\chi : \Gal(C/ \PP^1_{\F_q}) \to \mu_{q-1}$ such that $L(T,C/\PP^1_{\F_q},\chi)$ determines $U$ up to isomorphism.  The choice of character is not intrinsic to $U$ as it depends on distinguishing $P_1$ and $P_2$, and is furthermore only unique up to Frobenius.  But choosing $P_1$ and $P_2$ is sufficient to determine $U$ up to isomorphism as the complement of four points in $\PP^1_{\F_q}$ with two of them marked.  To remove the restriction on marking two points, one can consider the analogous question after permuting the ordering.
\end{Remark}

\begin{proof}
Let $\sigma$ be the fractional linear transformation sending $P_1$ to $\infty$, $P_2$ to $0$, and $P_3$ to $1$.  Let $\lambda = \sigma(P_4)$.  Then $\sigma$ induces an isomorphism between $C$ and $C_\lambda$ as covers of $\PP^1$ as both are ray class fields, and in particular identifies their Galois groups.  
Proposition~\ref{prop:distinguishedchar} gives the existence of a character $\chi$ with the required properties and shows it is unique up to $p$th powers as claimed.  In particular, there is a $0 \leq j < r$ such that
\[
L(T,\PP^1_{\F_q},\chi) = L(T,C_\lambda/\PP^1_{\F_q},\chi_{p^j,p^j}).
\]
By Proposition~\ref{prop:charcomputation}, we conclude that
\[
(L'(0,\PP^1_{\F_q},\chi) \mod{p}) = 2 \left( \frac{\lambda}{(\lambda-1)^2}   \right)^{p^j} .
\]
We can do the same for $U'$, obtaining analogous $\sigma'$, $\lambda'$, and $j'$.  

By hypothesis $(L'(0, \PP^1_{\F_q},\chi) \mod{p}) =( L'(0,\PP^1_{\F_q},\chi') \mod{p})^{p^i}$, so there is an integer $0\leq i' < r$ such that
\[
 \left( \frac{\lambda}{(\lambda-1)^2}   \right)^{p^{i'}} =   \frac{\lambda'}{(\lambda'-1)^2}   .
\]
As the function $\lambda \mapsto \lambda (\lambda-1)^{-2}$ is two-to-one and the $p$th power map is an automorphism of $\F_q$, we conclude that $\lambda' = \lambda^{p^{i'}}$ or $1/\lambda' = \lambda^{p^{i'}}$.  
Thus there is an automorphism $\beta$ of $\PP^1_{\F_q}$ fixing $0$ and $\infty$ and sending $\{1,\lambda\}$ to $\{1,\lambda'\}$;  $\beta$ is either the $p^{i'}$-th power map or the $p^{i'}$-th power map composed with the automorphism $x \mapsto x/\lambda$.
Then take $\alpha \colonequals (\sigma')^{-1} \circ \beta \circ \sigma$.
\end{proof}

%This is non-canonical as it depends on the ordering of the points and a choice of $\chi$.  
\begin{remark}
The automorphism $x \mapsto x/\lambda$ sends $\{1,\lambda\}$ to $\{1,1/\lambda\}$, and corresponds to switching the role of the third and fourth points.  
\end{remark}

\begin{remark} \label{remark:q9}
In the special case that $q=9$, we can distinguish $U = \PP^1_{\F_q} - \{0,1,\infty,\lambda\}$ using $L$-functions as follows.  Permuting the four marked points replaces $\lambda$ by $\lambda$, $1/\lambda$, $1-\lambda$, $1/(1-\lambda)$, $(\lambda-1)/\lambda$, and $\lambda/(\lambda-1)$.  Thus there are two choices of $U$ up to isomorphism, which occur when $\lambda=-1$ or when $\lambda \in \F_9 - \F_3$.  In light of Proposition~\ref{prop:charcomputation}, when $\lambda =-1$ we see that the value  $(L'(0,C_\lambda/\PP^1_{\F_q},\chi) \mod{p})$ lies in $\F_3$ for any character $\chi$.  Otherwise there is some character $\chi$ for which $(L'(0,C_\lambda/\PP^1_{\F_q},\chi) \mod{p})$ does not lie in $\F_3$.  Therefore we can distinguish these configurations of points using $L$-functions.  
\end{remark}

We finally prove a strong version of Conjecture~\ref{conj:lfun} for the projective line with four marked points.  As that conjecture deals with pointed curves, the modulus should be supported on three points.

\begin{Theorem} \label{thm:p1zilber}
Suppose that $q = p^r$ is odd and not equal to $9$, and let $U$ (resp. $U'$) be  
$\PP^1_{\F_q}$ with four rational points removed.   Fix one of the four points $M$ (resp $M'$) to define the Abel-Jacobi embedding, and let $\fm$ (resp. $\fm'$) be the sum of the other three.   Then let $X$ (resp. $X'$) be the ray class field of $\PP^1_{\F_q}$ in which $M$ (resp. $M'$) splits completely, and 
let $J_{\fm}$ (resp. $J_{\fm'}'$) be the generalized 
Jacobian. 

Suppose $\psi : J_{\fm}({\F}_q) \to J_{\fm'}'({\F}_q)$ is 
an isomorphism of groups.  If 
$$L(T,X'/\PP^1, \chi') = L(T,X/\PP^1, \chi' \circ \psi)$$ 
for all characters $\chi'$ of $J_{\fm'}(\F_{q})$, then 
$U$ and $U'$ are Frobenius twists of each other and the map
$\psi$ arises from a morphism of curves which sends $M$ to $M'$ and sends $\fm$ to $\fm'$.
\end{Theorem}

\begin{proof}
Using the $\PGL_2(\F_q)$-action on $\PP^1_{\F_q}$, we may and do assume that $M = M' = \infty$, $\fm = [0]+[1]+[\lambda]$, and $\fm' = [0]+[1]+[\lambda']$ with $\lambda, \lambda' \ne 0, 1$.  Note that $\lambda$ and $\lambda'$ are not canonical as they depend on the ordering of the removed points.
%Let $X$ and $X'$ be ray class fields of $\PP^1_{\F_q}$ with moduli $\fm$ and $\fm'$ in which $M$ and $M'$ split completely.  
Consider a character $\chi'$ of $\Gal(X'/\PP^1)$ satisfying conditions (1)-(4) of Theorem~\ref{thm:p1distinguished} (with $P_2 =0 $ say) and let $\chi = \chi' \circ \psi$.  Let $Y = X /\ker \chi$ and $Y' = X'/ \ker \chi'$.  We know that $Y' \to \PP^1$ is a cyclic cover of degree $q-1$ in which $\infty$ splits and that is ramified only above $0, 1,$ and $\lambda'$, and that the genus of $Y'$ is $(q-3)/2$.

Since the zeta functions of $Y$ and $Y'$ are equal by hypothesis, $Y$ also has genus $(q-3)/2$. We claim that $\chi$ also satisfies conditions (1)-(4) of Theorem~\ref{thm:p1distinguished} (with $P_2 \in \{0,1,\lambda\}$), which amounts to showing that $Y \to \PP^1_{\F_q}$ is totally ramified at two of $\{0,1,\lambda\}$ and has ramification index $(q-1)/2$ at the remaining point. 
To see this, let the ramification indices at the points of $Y$ above $0,1$, and $\lambda$ be $(q-1)/m_i$ (they divide $(q-1)$ since the cover is Galois of degree $q-1$).  The Riemann-Hurwitz formula implies that $q-5 = (q-1)(-2) + \sum_{i=1}^3 m_i((q-1)/m_i-1)$
which simplifies to $\sum m_i = 4$.  Thus two are $1$ and one is $2$.
Then the theorem follows from Proposition~\ref{prop:charcomputation}
%Corollary \ref{cor:identifyfibers}
 and Proposition~\ref{prop:distinguishedchar} as in the proof of Theorem~\ref{thm:p1distinguished}. 
\end{proof}

%
%\begin{remark}
%We can also distinguish the projective line with non-rational points removed using $L$-functions.  The simplest  approach is to just extend scalars so that all of the points become rational, and use an $L$-function for a character of the ray class field over $\F_{q^n}$ (with Galois group $\mu_{q^n-1}^2$). 
%\end{remark}

\subsection{Proof of Proposition~\ref{prop:charsdiffer}} \label{ss:technicalproof}

We continue to work over a finite field of size $q = p^r$, writing $k \colonequals (q-1)/2$ and assuming $p >2$.   We thank Z. Brady for help with the proofs of Lemma~\ref{lem:witness} and Proposition~\ref{prop:charsdiffer}; he gave us elegant proofs of the version of these results when $r=1$ (i.e. $q=p$).  Our extensions to the case when $r>1$ have become significantly less elegant, for which we apologize.

We begin with a couple of simple observations about the conditions $C(n)$ and $C(b,c)$ introduced in \eqref{eq:condn} and \eqref{eq:condbc}.

\begin{Lemma} \label{lem:condn}
We have:
\begin{enumerate}
\item \label{condn1}  condition $C(n)$ holds if and only if $C(k-n)$ holds;
\item \label{condn3} for $0< n < q-1$, the condition $C(n)$ is equivalent to condition $C(n,n)$.
\end{enumerate}
\end{Lemma}

\begin{proof}
The first is elementary.  The second was discussed in Remark~\ref{remark:kummer} and is a consequence of a classical theorem of Kummer. 
%For $0 < n < q-1$, note
%$$\alpha_{n,n} =\frac{(q-1)(q-2) \ldots (q-2n)}{n! \cdot n!} \equiv \pm \frac{(2n)!}{n! \cdot n!} \pmod{p}.$$
%We see $\alpha_{n,n}$ is non-zero modulo $p$ if and only the numerator and denominator have the same $p$-adic valuation.  %This is equivalent to number of multiples of $p^i$ appearing in $(2n)!$ being equal to twice the number of multiples of $p^i$ appearing in $n!$ for $i=1,2,\ldots, r-1$. 
%Since for any positive integer $i$ the number of multiples of $p^i$ less than or equal to $2n$ is at least twice the number of multiples of $p^i$ less than or equal to $n$, we conclude that $\alpha_{n,n} \not \equiv 0 \pmod{p}$ if and only if the number of multiples of $p^i$ appearing in $(2n)!$ is twice the number of multiples of $p^i$ appearing in $n!$ for $i=1,2,\ldots, r-1$.  This happens if and only if $( n \mod{p^i} ) < p^i/2$ for $i = 1 ,2, \ldots, r-1$.
\end{proof}

We next need some elementary yet tricky lemmas.  

\begin{Lemma} \label{lem:witness}
For any integer $b$ not congruent to $0,1,p,p^2,\ldots, p^{r-1}$ modulo $q-1$, there exists an integer $n$ satisfying $C(n)$ such that $(n b \mod{q-1}) \geq (q-1)/2$.
\end{Lemma}

\begin{proof}
Without loss of generality, we may assume that $1< b < (q-1)/2$.  If $b \geq k=(q-1)/2$ and $q \neq 3$, we take $n=1$ and notice that $C(1)$ holds if $q \neq 3$.  (If $q=3$, there do not exist any $b$ as in the statement of the lemma.)  We proceed by induction on $r$, but treating some special cases by hand.  In this proof only, we will use the notation $C_i(n)$ to denote the condition in \eqref{eq:condn} with $q$ replaced by $p^i$.

For the base case $r=1$, let $n = \lfloor (k-1)/b \rfloor +1$, which satisfies $n < (p-1)/2$ as $b >1$.  We directly see that $nb > (k-1) = (q-1)/2-1$ and that $n b \leq (k-1) + b < q-1$ as $b < (q-1)/2$.

Next, suppose that $b \not \equiv 0,1,p,p^2,\ldots, p^{r-2} \pmod{p^{r-1}-1}$.  By induction, we may suppose that there exists an integer $n_0$ satisfying $C_{r-1}(n_0)$ such that there are integers $a,c$ for which 
\[
n_0 b = a (p^{r-1}-1) + c  \quad \text{with} \quad \frac{p^{r-1}-1}{2} \leq c < p^{r-1}-1.
\]
As $C_{r-1}(n_0)$ holds, we know that $n_0 < (p^{r-1}-1)/2$ and hence
\begin{equation*} 
a = \frac{n_0 b -c}{p^{r-1}-1} < \frac{b-1}{2}. 
\end{equation*}
Then $a(p-1) < b (p-1)/2 - (p-1)/2$ and $(p^r-1)/2 - (p-1)/2 \leq pc$ and hence 
\begin{equation} \label{eq:inequality1}
-b \cdot \frac{p-1}{2} +  \frac{p^r-1}{2} < -(p-1)a + pc  . %\leq pc < p^r-1.
\end{equation}
Let $\epsilon$ be the smallest positive integer such that
\[
\frac{p^r-1}{2} \leq  -(p-1)a + pc  + \epsilon b.
\]
Then \eqref{eq:inequality1} shows that $\epsilon \leq (p-1)/2$.  
Therefore $n \colonequals p n_0 + \epsilon$ satisfies $C_r(n)$.  We compute 
\[
nb = (p^r-1) a - (p-1)a + pc + \epsilon b.
\]
By the choice of $\epsilon$ and as $b < (p^{r}-1)/2$, we conclude that 
\[
\frac{p^r-1}{2} \leq -(p-1)a + pc  + \epsilon b \leq \frac{p^r-1}{2} +b < p^r-1.
\]
Thus $(nb \pmod{p^{r}-1} ) \geq (p^r-1)/2$ as desired.  

Finally, we directly treat the cases that $b \equiv 0,1,p,p^2,\ldots, p^{r-2} \pmod{p^{r-1}-1}$ which the above inductive step does not address.
We first consider the case when $b = a (p^{r-1}-1)$ with $a \in \{1,2,\ldots,p\}$.  If $a=1$, take $n=p$.  Otherwise take $n = \lfloor (p-1)/(2a) \rfloor +1$.  For both, we see that $n$ satisfies $C_r(n)$ and $(nb \mod{p^r-1}) \geq (p^r-1)/2$.

In the other cases, write $b = p^i + a (p^{r-1}-1)$ with $i \in \{0,1,\ldots,r-2\}$ and $0<a\leq p$.  

\begin{description}
\item[The case $i=0$ and $a>1$]  First suppose $b = 1 + a (p^{r-1}-1)$ with $a>1$.  If $a<p/2$, take $n= p$, and observe that $nb \equiv p + a(1-p) \pmod{p^r-1}$ so $(nb \mod{p^r-1}) \geq (p^r-1)/2$.  If $a > p/2$, take $n=1$ and observe that $b \geq (p^r-1)/2$.  In both cases $C_r(n)$ holds. 

\item[The case $a=1$ and $i>0$]  Now suppose $b = p^i + (p^{r-1}-1)$ with $i>0$.  Taking $n=c p^{r-i-1} + c $,  we compute that 
\[
nb \equiv 2c p^{r-1} + c p^i - c - c p^{r-i-2}(p-1) \mod{p^r-1}.
\]
If we choose $c = (p-1)/2$, then we see that
\[
2c p^{r-1} + c p^i - c - c p^{r-i-2}(p-1) \leq 2 c p^{r-1} + c p^i \leq (p-1)p^{r-1} + \frac{p-1}{2} p^i \leq p^r -1.
\]
We also compute that
\[
2c p^{r-1} + c p^i - c - c p^{r-i-2}(p-1)  \geq (p-1) p^{r-1} - \frac{p-1}{2} - (p-1)^2 p^{r-i-2}/2 \geq (p^r-1)/2.
\]
Thus $(nb \mod{p^r-1}) \geq (p^r-1)/2$ and we directly see that $C_r(n)$ holds.

\item[The case $i>0$ and $a>1$]  Now suppose $ b=p^i + a(p^{r-1}-1)$ with $a>1$ and $i >0$.  If $a > p/2$, take $n=1$.  Otherwise, choose $a' = \lceil (p-1)/(2a) \rceil$ and let $n = p^{r-i-1} + a'$.  Notice that $C_r(n)$ holds as $a>1$, and that $(p+1)/2 \leq aa'+1 \leq p-1$.  We compute that
\[
nb \equiv p^{r-1} + a a' (p^{r-1}-1) + a p^{r-i-2}(1-p) + a' p^i \mod{p^r-1}.
\]
 As $r-i-2< r-2$ and $i < r-1$, we see that the dominant term in the above expression is $(1 + a a') p^{r-1}$.  In particular, 
\[
p^{r-1} + a a' (p^{r-1} -1)+ a p^{r-i-2}(1-p) + a' p^i \leq (a a' +1 ) p^{r-1}+ a' p^i < p^{r}-1
\]
and
\[
p^{r-1} + a a' (p^{r-1}-1) + a p^{r-i-2}(1-p) + a' p^i \geq (a a' +1 )p^{r-1} + a p^{r-i-2} (1-p) \geq (p^r-1)/2
\]
as $a < p/2$.  
Thus $(nb \mod{p^r-1}) \geq (p^r-1)/2$ as desired.
\end{description}
Our argument has covered all of the cases, as we have assumed $b$ is not congruent to $0,1,p,p^2,\ldots, p^{r-1}$ modulo $p^r-1$.
\end{proof}

\begin{Lemma} \label{lem:equalpowers}
Suppose $b = p^{m_1}$ and $c = p^{m_2}$ with $ 0 \leq m_1 < m_2 < r$.  For odd $q \neq 9$, there exists an integer $n$ for which $C(n)$ does \emph{not} hold and for which $C(nb,nc)$ holds.  
\end{Lemma}

\begin{proof}
Take $n = (p+1)/2$.  Notice that $ p/2 < n < p$, so $C(n)$ does not hold.  On the other hand,
\begin{align*}
( n b \mod{2k}) + (n c \mod{2k}) & = \frac{p+1}{2} ( p^{m_1} + p^{m_2}) \\
& \leq \frac{p+1}{2} ( p^{r-1} + p^{r-2}) \\
& \leq p^r \left( \frac{1}{2} + \frac{1}{p} + \frac{1}{2p^2} \right) < q-1.
\end{align*}
The last inequality uses that $p>3$ or that $p=3$ and $r>2$.  Finally, notice that 
\[ \alpha_{nb,nc} \equiv \pm \binom{ (p+1)/2 \cdot p^{m_1} + (p+1)/2 \cdot p^{m_2} }{(p+1)/2 \cdot p^{m_1}} \pmod{p}
\]
and that this binomial coefficient is non-zero modulo $p$ by Kummer's theorem on binomial coefficients \cite{kummer}.
%Finally we use that $\displaystyle \alpha_{nb,nc} \equiv \pm \frac{ \left( (p+1)/2 \cdot p^{m_1} + (p+1)/2 \cdot p^{m_2} \right)!  }{ \left( (p+1)/2 \cdot p^{m_1} \right)!  \left( (p+1)/2 \cdot p^{m_2} \right)! }  \pmod{p}$.  For $i \leq m_1$, there $(p+1)/2 \cdot p^{m_1 - i } + (p+1)/2 \cdot p^{m_2 - i }$ multiples of $p^i$ appearing in the factorials comprising the denominator of $\alpha_{nb,nc}$, and a similar number appearing the numerator.  For $m_1 < i \leq m_2$ there are $(p+1)/2 \cdot p^{m_2 - i }$ multiples of $p^i$ appearing in the denominator and a similar number in the numerator.  There are no multiples of $p^i$ appearing when $i > m_2$.  Thus the numerator and denominator have the same $p$-adic valuation, so $\alpha_{nb,nc} \not \equiv 0 \pmod{p}$.  
\end{proof}

We are now ready to prove Proposition~\ref{prop:charsdiffer}; we again thank Z. Brady for the key idea.

\begin{proof}[Proof of Proposition~\ref{prop:charsdiffer}]
As $b$ is odd, for any integer $n$
\begin{equation} \label{eq:congruence}
A(b,n) := (nb  \mod{2k}) + ((k-n) b  \mod{2k}) \equiv k \pmod{2k}.
\end{equation}
In particular, $A(b,n)  > 0$.

If $C(n)$ holds, then $C(k-n)$ holds by Lemma~\ref{lem:condn}(\ref{condn1}).  By hypothesis 
$C(nb,nc)$ and $C((k-n) b, (k-n)c)$ hold as well, so 
\begin{align*}
0<A(b,n) + A(c,n) &= (nb  \mod{2k}) + ((k-n) b  \mod{2k}) \\
&+   (n c \mod{2k}) + ((k-n) c \mod{2k}) < 4k.
\end{align*}
As $A(b,n) + A(c,n)  \equiv k + k \equiv 0 \pmod{2k}$ 
and $0 < A(b,n) + A(c,n)$, we conclude that $A(b,n) + A(c,n) = 2k$ if $C(n)$ holds.

We next claim that if $C(n)$ holds then $(n b \mod{2k}) < k$.  If $(n b \mod{2k}) > k$, then \eqref{eq:congruence} shows that $A(b,n) > 2k$.  This is impossible as we know that $A(b,n) + A(c,n) = 2k$ and each term is non-negative.  Since $b$ is relatively prime to $2k$ and $n <k$ as $C(n)$ holds, the claim follows.

If $b$ were not a power of $p$ or $1$, Lemma~\ref{lem:witness} would give an integer $n$ for which $C(n)$ holds and such that $(nb \mod{2k}) \geq k$, contradicting the claim.  Thus $b \in \{1,p,p^2,\ldots, p^{r-1}\}$.  Likewise, we see that $c \in \{1,p,p^2,\ldots, p^{r-1}\}$.  Then Lemma~\ref{lem:equalpowers} shows that either $b=c$ or we are in the exceptional case with $q=9$.  This completes the proof.
\end{proof}

\section{General Affine Curves}

Given an equation $F(x,y) = 0$ defining a smooth projective curve $C/\F_q$, our goal is to recover the coefficients
of the equation from an explicit set of $L$-functions for covers of $C$.  We will do so using Artin-Schreier covers.  Let $K$ be the function field of $C$ and
$p$ the characteristic of $K$.

In what follows, we will deal with extensions $\F_{q^m}/\F_q$ and consider 
both the relative trace $\Tr_m: \F_{q^m} \to \F_q$ and the absolute trace
down to the prime field $\Tr: \F_{q^m} \to \F_p$.

\subsection{Exponential Sums} \label{ss:exp} We begin by reviewing the connection between $L$-functions for Artin-Schreier extensions and exponential sums; a standard reference is \cite[VI]{Bom}.  

Artin-Schreier extensions are cyclic degree $p$ extensions of $K$, and are given by the equations $z^p -z =f$ where $f \in K$ is not equal to $w^p - w$ for any $w \in K$.  For a fixed $f$, regular on an open subset $U$ of $C$, the conductor of $E/K$ (where $E=K(z)$, with $z^p-z=f$) is bounded above by $(f)_\infty$, the divisor of poles of $f$ on $C$.  

There is a character $\chi_f : \Gal(E/K) \to \C^\times$
such that $\chi_f(\Frob_P) = \exp(2\pi i \Tr(f(P))/p)$ for $P \in U$,
where $\Tr$ is the absolute trace $\Tr: \F_{q^m} \to \F_p$. 
We assume that $E/K$ is ramified outside $U$ and consider the exponential sums
$$S_m(\chi_f) = S_m(f) \colonequals \sum_{P \in U(\F_{q^m})} \exp(2\pi i \Tr(f(P))/p).$$
The values of these sums are determined by $L(T,C,\chi_f)$.

When $p \neq 2$, taking $\varpi = 1 - \exp(2\pi i /p)$ we note that 
$\Z[\exp(2\pi i /p)]/(\varpi)$ is isomorphic to $\F_p$ and that
\begin{equation}
S_m(f) = \sum_{P \in U(\F_{q^m})}(1-\varpi)^{\Tr(f(P))} \equiv 
|U(\F_{q^m})| - \left( \sum_{P \in U(\F_{q^m})}{\Tr(f(P))} \right)\varpi \pmod{\varpi^2}.
\end{equation}

If $p=2$, we instead consider the exponential sums 
$$S_m({\chi_{\f}}) = \sum_{P \in U(\F_{q^m})} e^{2\pi i T({\f}(P))/4}$$
where $\f = (f,0)$ is a Witt vector of length $2$ and
$T:W_2(\F_{q^m}) \to W_2(\F_2) \isom \Z/4\Z$ is the absolute trace (see \cite{VW}).
We can form an $L$-function as before, which will be a factor of the
zeta function of an Artin-Schreier-Witt cover of $C$. Letting $\varpi = 1 - i$, we
get $S_m({\f}) \equiv |U(\F_{q^m})| - (\sum_{P \in U(\F_{q^m})}{\Tr(f(P))})\varpi \pmod{\varpi^2}$.

\begin{Proposition} \label{prop:determinesums}
Let $\gamma_1, \ldots, \gamma_u$ be a basis for $\F_q$ over $\F_p$.
The $L$-function $L(T,C,\chi_{\gamma_i f})$ determines the value of
\begin{equation}
T_m(f):=\sum_{P \in U(\F_{q^m})}{\Tr_m(f(P))} \in \F_q
\end{equation}
for each $m$, where $\Tr_m: \F_{q^m} \to \F_q$ is the relative trace.
\end{Proposition}

\begin{proof}
We have seen that the knowledge of the L-functions $L(T,C,\chi_{\gamma_i f})$ provides us with 
\[\sum_{P \in U(\F_{q^m})}{\Tr(\gamma_i f(P))} \in \F_p\]
 for each $m$.  But these determine $T_m(f)$ as the trace pairing is non-degenerate.
\end{proof}

\subsection{Recovering Affine Equations from Exponential Sums}

We now show how to use $L$-function for appropriate Artin-Schreier covers of a curve to recover the coefficients of a defining equation via considering exponential sums.  

\begin{Theorem} 
\label{thm:equations}
For fixed odd $q$ and $d \ge 1$, there is an explicit finite
set of polynomials $f \in \F_q[x,y]$ such that we can
recover the coefficients of any absolutely irreducible $F \in \F_q[x,y]$ of degree $d$ 
from the $L$-functions of the Artin-Schreier extensions $z^p-z=f$ of the function field $K$ for the curve determined by $F=0$.
\end{Theorem}

\begin{remark} 
When $q$ is even we would need the Artin-Schreier-Witt extensions discussed in Section~\ref{ss:exp}. Our proof does not apply as written, but we expect it to adapt.
\end{remark}

The proof will use the following lemma about the values of symmetric functions and power sums.  For a positive integer $n$, let $e_1,\ldots e_n$ be the elementary symmetric functions in $x_1,\ldots, x_n$ viewed as elements of $\Z[x_1,\ldots,x_n]$, and define
\[
p_k = \sum_{i=1}^n x_i^k \in \Z[x_1,\ldots,x_n].
\]

\begin{Lemma} \label{lem:powers}
Fix a field $\kk$ of characteristic $p$ and an integer $n>0$. 
Given $\alpha_1,\alpha_2, \ldots  \alpha_r \in \kk$ and multiplicities $n_1, n_2, \ldots , n_r$ with $1 \leq n_i <p$ and $n=\sum_{i=1}^r n_i$, let $\beta_1, \ldots \beta_n$ be these $n$ values, with $\alpha_i$ occurring $n_i$ times.  The sequence of values of the power sums
\[
p_j(\beta_1,\ldots, \beta_n) = \sum_{i=1}^r n_i \alpha_i^j
\]
for all $j$ uniquely determine the values of the symmetric functions $e_j(\beta_1,\ldots, \beta_n)$ for $1 \leq j \leq n$.  If the field $\kk$ is a finite field with $q$ elements, it suffices to know the values of the power sums for $j=1 ,\ldots, q-1$.
\end{Lemma}

We defer the proof to the next subsection.  

\begin{Lemma} \label{lem:h}
There exists an integer $m > d$ with $p \nmid m$ (depending on $q$ and $d$ only) and an irreducible polynomial $h \in \F_q[x]$ of degree $m$ such that all solutions to $F(x,y)=h(x)=0$ are defined over $\F_{q^m}$ and that there are $m \deg_y F$ of them.  Whether a particular $h$ works can be verified using exponential sums on $C$.
\end{Lemma}

\begin{proof}
Using Chebotarev, for sufficiently large $m$ there exists an irreducible polynomial $h \in \F_q[x]$ of degree $m$ 
such that all solutions of 
$F(x,y)=h(x)=0$ are defined over $\F_{q^m}$ and there are
$m\deg_y F$ of them.

To check this condition for a given irreducible $h \in \F_q[x]$, run through all $s \in \F_q[x]$ with $\deg s < m$ and
take $f=\left(1-h(x)^{q^m-1}\right) \left(1-(y-s(x))^{q^m-1} \right)$.  Notice that for $a,b \in \F_{q^m}$ we have $f(a,b)\neq 0$ if and only if $h(a) = 0$ and $b=s(a)$.  When $f(a,b)$ is non-zero, it takes on the value $1$.  If there exists $a \in \F_{q^m}$ such that $h(a) = 0$ and $F(a,s(a))=0$, it follows that $F(a^{q^j},s(a^{q^j}))=0$ for $j=1,\ldots,m-1$, and hence $F(a,s(a))=0$ for all roots $a$ of $h$ in $\F_{q^m}$.  In this case we conclude that $T_m(f) = m^2$ and we have $m$ solutions to $h(x) = F(x,y)=0$ over $\F_{q^m}$ given by $(a^{q^j}, s(a^{q^j}))$ for $j=0 \ldots m-1$.  If not, then  $T_m(f)=0$. 

By varying $s$ so that it takes on different values at the roots of $h$, we can detect whether there are $m \deg_y F$ solutions to $F(x,y)=h(x)=0$ and that are defined over $\F_{q^m}$.
\end{proof}

\begin{Lemma} \label{lem:f}
Choose $h$ as in Lemma~\ref{lem:h}.  For any fixed $i,j$, there exists $f \in \F_q[x,y]$ such that for $(a,b) \in \F_{q^m}^2$:
\begin{enumerate}
\item \label{cond:ram}  $z^p -z = f$ is irreducible and the Artin-Schreier extension is ramified over every infinite point of $C$;

\item  if $h(a) \neq 0$ then $f(a,b) = 0$;

\item  if $h(a)=0$ then $f(a,b) = a^i b^j$.
\end{enumerate}
The value of $T_m(f)$ can be computed without knowing an explicit choice of $f$.
\end{Lemma}

\begin{proof}
Consider possible $f$ of the form $f = x^iy^j(1-h^{q^m-1}) (1+(x^{q^m}-x)g)$, for an arbitrary $g \in \F_q[x,y]$.  The second and third statements are straightforward to check, and it is clear that the value $f(a,b)$ for $a,b \in \F_{q^m}$ is independent of the choice of $g$.  Thus we must show there is a choice of $g$ such that $z^p-z=f$ is irreducible and the Artin-Schreier extension is ramified over every infinite point of $C$.

Take a large integer $D$ and consider all functions on $K$ represented by a
$g \in \F_q[x,y]$ with $\deg g \le D$. There are at least $q^{D\cdot d }$ such functions (where $d = \deg F$ as before). For any such $g$, the resulting $f$ is a polynomial of degree at most $D+O(1)$ and is determined, up to a constant, by the polar parts at the places at infinity.  If condition (\ref{cond:ram}) is not satisfied, we must have that the polar part at one of the places at infinity is of the exceptional form $w^p-w$ for some  Laurent polynomial $w$ at that place. 
Let $v$ denote the valuation at that place and 
$m \colonequals \min\{v(x),v(y)\}$; note that $0 > m  \ge -d/\deg v$. 
Then $v(f) \ge m D$ and if the polar part of $f$ at $v$ is $w^p -w$
then $v(w) \ge mD/p$ and the coefficients of the powers $\pi^j$ of a local
parameter $\pi$ at $v$ occurring in the expansion of $f$ for $j < mD/p$
and $p \nmid j$ must vanish.  For $f$ to be of this exceptional form at the place $v$
 results in at least $D(1-1/p)^2$ independent linear
conditions on $f$. As there are at most $d$ 
such places, we see that
there are at most $O(q^{D(1-(1-1/p)^2)})$ polynomials $f$ of degree $D$ for which condition (\ref{cond:ram}) does not hold.  

Thus for $D$ sufficiently large (depending just on $d$ and $p$), we see that at least half of the choices of $g$ must give an $f$ satisfying condition (\ref{cond:ram}).  We conclude that the value of $T_m(f)$ for a choice of $f$ satisfying the first condition is the most common value obtained by computing $T_m ( x^i y^j (1-h^{q^{m}-1}) (1 + (x^{q^m}-x)g))$ for
all choices of $g$ with $\deg g \le D$ and some suitably large $D$.
\end{proof}

\begin{proof}[Proof of Theorem \ref{thm:equations}]
By Proposition~\ref{prop:determinesums}, the values of the exponential sums on $C$ are determined by the knowledge of the $L$-function.  
Choose $m$ and $h$ as in Lemma~\ref{lem:h} using our knowledge of the degree of $F$ and of the values of the exponential sums on $C$.  Then choose $f$ as in Lemma~\ref{lem:f}.  We compute that
\[
T_m(f) = \sum_{\stackrel{(a,b)\in (\F_{q^m})^2}{F(a,b)=0}} \Tr_m(f(a,b)) = \sum_{a:h(a)=0} \sum_{b:F(a,b)=0} \Tr_m(a^i b^j)  
\]
Since Frobenius cyclically permutes the roots of $h(x)$, we see that
\[
T_m(f) = m\sum_{a:h(a)=0} \sum_{b:F(a,b)=0} a^i b^j.
\]
As $p \nmid m$, we can use these values with $i=0,1,\ldots,m-1$ to recover, for each root $a$ of $h$, the value of $\sum_{b:F(a,b)=0} b^j$ as the Vandermonde matrix is invertible.  By Lemma~\ref{lem:powers}, these values determine the polynomial $F(a,y) \in \F_q[y]$ for each root $a$.  Viewing $F(x,y)$ as a polynomial in $x$ with coefficients in $\F_q[y]$,  the polynomials $F(a,y)$ let us determine $F(x,y) \in \F_q[x,y]$ as $m>d$.  This completes the proof.
\end{proof}

\begin{Remark}
The proof uses a ridiculously large number of $L$-functions of Artin-Schreier covers, something like $u(q^{m^2}+q^D)$, where $u=[\F_q:\F_p]$, $m$ comes from Lemma \ref{lem:h} and can be explicitly computed applying the effective Chebotarev Theorem from \cite{MS} to the splitting field of $F(x,y)$ over $\F_q(x)$ and $D$ comes from Lemma \ref{lem:f} and can be extracted from the proof there.  It would be nice to have a smaller such set.  In specific situations, it is often possible to be more efficient.
\end{Remark}

\begin{Example} \label{ex:legendre}
Let us use $L$-functions to distinguish members of the Legendre family of elliptic curves $y^2 = x ( x-1)(x-\lambda)$.  We take $g(x,y) = x (1-y^{q-1})$.  A direct calculation as in the proof shows that
\[
T_1(g) = \sum_{b : b(b-1)(b-\lambda) =0} b = 1 + \lambda.
\]
Thus the value of this single exponential sum determines $\lambda$.

We now explain how this is a simplification of the proof of Theorem~\ref{thm:equations} for this particular family of curves.  
First, note that it suffices to recover the coefficient of $y$ in the polynomial $F(x,y) = x^2 - y(y-1)(y-\lambda)$, so we will focus just on recovering the linear term of $F(0,y)$ instead of the whole polynomial.  (Note that we swap the usual role of $x$ and $y$ to conform to the notation in our proof.)  
This polynomial has degree $d=3$.  The assumption in Lemma~\ref{lem:h} that $m>d$ is irrelevant as we are only interested in the linear term, so let us take $m=1$.  Then $h(x) = x$ satisfies the other conditions of Lemma~\ref{lem:h} as $0 = x (x-1)(x-\lambda)$ has three solutions over $\F_q$.  Choosing $(i,j) = (0,1)$, we claim that $f = y (1-x^{q-1})$ satisfies the conditions of Lemma~\ref{lem:f}.  As $f$ has a pole of order $3q-1$ at infinity, $z^p -z = f$ is irreducible and ramified over the unique point at infinity.  If $h(a)=a \neq 0$, then $f(a,b) = b (1 -a^{q-1})=0$ and otherwise $f(a,b) = b$ as desired.  (Having knowledge of the behavior of $x$ and $y$ at infinity makes finding a suitable $f$ much easier.)  This recovers the exponential sum we used above,
\[
T_1(f) = \sum_{b : b(b-1)(b-\lambda) =0} b = 1 + \lambda.
\]

  The proof of Theorem~\ref{thm:equations}, which recovers the whole equation, is more complicated as it needs a larger $m$; it recovers the coefficient of $y$ in $F(0,y) = -y(y-1)(y-\lambda)$ using Artin-Schreier covers for various $(i,j)$.  In particular, it recovers the coefficient from recovering the sums of powers of the roots $0$, $1$, and $\lambda$.  These power sums in turn determine $F(0,y)$.
\end{Example}

\subsection{Power Sums and Symmetric Functions}
It remains to prove Lemma \ref{lem:powers}.  Continuing the notation of the lemma, we have the following elementary recurrences:
\begin{align}
k e_k & = \sum_{i=1}^k (-1)^{i-1} e_{k-i} p_i \quad \text{ when } 1 \leq k \leq n  \label{eq:recurrence1}\\
0 & = \sum_{i=k-n}^k (-1)^{i-1} e_{k-i} p_i \quad \text{ when } k>n \label{eq:recurrence2}.
\end{align}
These can be deduced by considering the generating function
\[
F(T) = \prod_{i=1}^n (1- x_i T) = \sum_{i=0}^n (-1)^i e_i T^i.
\]
The logarithmic derivative is
\[
\frac{F'(T)}{F(T)} = \sum_{i=1}^n \frac{-x_i}{1 - x_i T} = \sum_{j=0}^\infty  - p_{j+1} T^j.
\]
The recurrences follow by cross-multiplying.  

For a field $K$ of characteristic zero, $K[p_1,\ldots,p_n ] = K[e_1,\ldots,e_n] = K[x_1,\ldots, x_n]^{S_n}$.  This is false in positive characteristic (and over $\Z$).

\begin{example} \label{ex:power2}
Let $p=2$ and $n=2$.  Then $p_1 = e_1$, and the recurrences tell us that $2 e_2 = e_1 p_1 + p_2 = 0$ and that
\[
p_3 = e_1 p_2 - e_2 p_1  \quad \text{i.e.} \quad  x_1^3 + x_2^3  =(x_1 + x_2) \left( x_1^2 + x_2^2  - x_1 x_2. \right)
\]
Thus we see that
\[
e_2 = \frac{p_1 p_2 - p_3}{p_1} \quad \text{i.e.} \quad x_1 x_2 = \frac{ (x_1+x_2)(x_1^2 + x_2^2) - x_1^3 - x_2^3}{x_1 + x_2}.
\]
So we can recover the symmetric polynomials as rational functions in $p_1, p_2, p_3, \ldots$.  (Actually, we just need $p_1$ and $p_3$, as $p_1^2 = p_2$.)

Evaluating at elements of $\kk$, we can recover the values of $e_1$ and $e_2$ provided $p_1(x_1,x_2) \neq 0$.  This restriction is necessary: consider the family $x_1 = \alpha$, $x_2 = - \alpha = \alpha$ for $\alpha \in \kk$, where we have $p_i(x_1,x_2) =0$ for all $i$ but $e_2(x_1,x_2) = - \alpha^2$.
\end{example}

\begin{lem} \label{lem:field}
For any field $\kk$, the fields $\kk(p_1,\ldots,p_n, \ldots )$ and $\kk(e_1,\ldots, e_n)$ are equal as subfields of $\kk(x_1,\ldots,x_n)$.
\end{lem}

\begin{proof}
As the $p_i$ are symmetric, they may be written as polynomials in terms of the $e_i$'s.  So it suffices to write the $e_i$ as rational functions in terms of the $p_i$.   The key idea is to show that the power series
\[
\sum_{j=0}^\infty - p_{j+1} T^j
\]
is a rational function with coefficients in $k(p_1,p_2,\ldots)$.  We already know it is a rational function with coefficients in $k(x_1,\ldots ,x_n)$ from the identity 
\[
\frac{F'(T)}{F(T)} = \sum_{j=0}^\infty - p_{j+1} T^j.
\] 
As $F(T)$ has no repeated roots, $F(T)$ and $F'(T)$ have no common roots and so if
\[
\frac{A(T)}{B(T)} = \frac{F'(T)}{F(T)} 
\]
with $A(T), B(T) \in k(x_1,\ldots,x_n)[T]$ and $A(T)$ of degree at most $n-1$ (in T) and $B(T)$ of degree $n$ (in T), then $A(T) = g F'(T)$ and $B(T) = g F(T)$ for some non-zero $g \in k(x_1,\ldots ,x_n)$.  If we require that the constant term of $B(T)$ is one, then $A(T) = F'(T)$ and $B(T) = F(T)$.

Now write $A(T) = \sum_{i=0}^{n-1} a_i T^i$ and $B(T) = \sum_{i=0}^n b_i T^i$, viewing the $a_i$ and $b_i$ as variables, and consider 
\[
A(T) = B(T) \cdot \left( \sum_{j=0}^\infty - p_{j+1} T^j \right).
\]
Equating coefficients of $T$, we obtain an (infinite) system of linear equations with variables $\{a_i\}$ and $\{b_i\}$ and coefficients in $\kk(p_1,p_2,\ldots)$.  This is solvable over the larger field $\kk(x_1,\ldots,x_n)$ as we have the identity of rational functions 
\[
\frac{F'(T)}{F(T)} = - \sum_{j=0}^\infty - p_{j+1} T^j.
\]
Therefore, the system of linear equations is also solvable over $\kk(p_1,p_2,\ldots )$. 

In particular, there exists $A(T)$ of degree at most $n-1$ and $B(T)$ of degree $n$ in $\kk(p_1,p_2,\ldots)[T]$ such that 
\[
\frac{A(T)}{B(T)} = \sum_{j=0}^\infty - p_{j+1} T^j.
\]
We may certainly scale by an element of $\kk(p_1,p_2,\ldots)$ so the constant term of $B(T)$ is $1$, in which case we must have that $A(T) = F'(T)$ and $B(T) = F(T)$.  Therefore the coefficients of $F(T)$, which up to sign are the elementary symmetric functions, lie in $\kk(p_1,p_2,\ldots)$.  This completes the proof. 
\end{proof}

\begin{remark}
It follows from Lemma~\ref{lem:field} that $\kk(p_1,p_2,\ldots)$ is generated by finitely many of the $p_i$'s, although it is not clear which ones do so.  As Example~\ref{ex:power2} shows, it is not possible to simply write $e_1,\ldots e_n$ in terms of $p_1,\ldots , p_n$.  
\end{remark}

We can now prove Lemma~\ref{lem:powers}.

\begin{proof}[Proof of Lemma~\ref{lem:powers}]
Evaluate $F$ at $x_1 = \beta_1, x_2 = \beta_2, \ldots, x_n = \beta_n$, obtaining
\[
G(T) := \prod_{i=1}^n (1 - \beta_i T ) = \sum_{i=0}^n (-1)^i e_i(\beta_1,\ldots , \beta_n) T^{i}.
\]
The logarithmic derivative is
\[
\frac{G'(T)}{G(T)} = \sum_{i=1}^r \frac{ - n_i \alpha_i}{ 1 - \alpha_i T} = - \sum_{j=0}^\infty p_{j+1}(\beta_1,\ldots,\beta_n) T^j.
\]
So knowledge of the power sums determines the rational function $\frac{G'(T)}{G(T)}$.  Note that its partial fraction decomposition is unique.

First assume that none of the $\alpha_i$ are zero.  The poles of this rational function are the $\alpha_i$, and the residue at $\alpha_i$ determines $n_i$ modulo $p$.  This is enough to recover the values of the symmetric functions.  
Otherwise, without loss of generality assume $\alpha_1 = 0$.  Then we recover $\alpha_2,\ldots, \alpha_r$ and $n_2,\ldots, n_r$, and can find $n_1$ via $n_1 = n - (n_2 + \ldots + n_r)$.
\end{proof}

\begin{remark}
Without the restriction that $1 \leq n_i <p$, it is not possible to recover the multiset $\{ \beta_1,\ldots, \beta_n\}$, as one could add $p$ to one of the $n_i$ and subtract $p$ from another without changing the values of the power sums or the logarithmic derivative of $G(T)$.  %We see this in the examples.

Lemma~\ref{lem:powers} is computationally effective.  If all we know is $n$ and the power sums, we can use a continued fraction algorithm on the power series to find the rational function.  It is not immediately clear how many power sums are needed in general. If the field $\kk$ is the finite field of $q$ elements, then $p_{q-1+i}=p_i$ and thus it's enough to
have $p_i$ for $i=1,\ldots,q-1$.
\end{remark}

%----------------------------------------------------------------------
%\end{thebibliography}
%----------------------------------------------------------------------


\begin{bibdiv}

\begin{biblist}

%\bib{AT}{book}{
%AUTHOR = {Artin, Emil},
%AUTHOR = {Tate, John},
%     TITLE = {Class field theory},
%      NOTE = {Reprinted with corrections from the 1967 original},
% PUBLISHER = {AMS Chelsea Publishing, Providence, RI},
%      YEAR = {2009},
%     PAGES = {viii+194},
%}

\bib{BKT}{article}{
   author={Bogomolov, Fedor},
   author={Korotiaev, Mikhail},
   author={Tschinkel, Yuri},
   title={A Torelli theorem for curves over finite fields},
   journal={Pure Appl. Math. Q.},
   volume={6},
   date={2010},
   number={1, Special Issue: In honor of John Tate.},
   pages={245--294},
   issn={1558-8599},
   
}


\bib{Bom}{article}{
AUTHOR = {Bombieri, Enrico},
     TITLE = {On exponential sums in finite fields},
   JOURNAL = {Amer. J. Math.},
  %FJOURNAL = {American Journal of Mathematics},
    VOLUME = {88},
      YEAR = {1966},
     PAGES = {71--105},
     }


\bib{BV}{article}{
author={Booher, Jeremy},
author={Voloch, Jos\'e Felipe},
title={Recovering algebraic curves from L-functions of Hilbert class fields},
journal={Res. Number Theory},
   volume={6},
   date={2020},
   number={4},
   pages={43},
}


%\bib{CJ}{article}{
%   author={Conway, J. H.},
%   author={Jones, A. J.},
%  title={Trigonometric Diophantine equations (On vanishing sums of roots of unity)},
%   journal={Acta Arith.},
%  volume={30},
%  date={1976},
%  number={3},
%  pages={229--240},
%}

\bib{CDL+}{article}{
AUTHOR = {Cornelissen, Gunther},
author={de Smit, Bart},
author={Li, Xin},
author={Marcolli, Matilde}, 
author={Smit, Harry},
     TITLE = {Characterization of global fields by {D}irichlet {$L$}-series},
   JOURNAL = {Res. Number Theory},
    VOLUME = {5},
      YEAR = {2019},
    NUMBER = {1},
     PAGES = {Art. 7, 15},
}


\bib{gassmann}{article}{
    AUTHOR = {Ga{\ss}mann, F.},
     TITLE = {Bemerkungen zur Vorstehenden Arbeit von Hurwitz: \"{U}ber {B}eziehungen zwischen den {P}rimidealen eines  algebraischen {K}\"{o}rpers und den {S}ubstitutionen seiner
              {G}ruppe},
   JOURNAL = {Math. Z.},
  FJOURNAL = {Mathematische Zeitschrift},
    VOLUME = {25},
      YEAR = {1926},
    NUMBER = {1},
     PAGES = {661--675},
      ISSN = {0025-5874},
}

\bib{hk}{article}{
   author={Halter-Koch, Franz},
   title={A note on ray class fields of global fields},
   journal={Nagoya Math. J.},
   volume={120},
   date={1990},
   pages={61--66},
   issn={0027-7630},
}

\bib{kummer}{article}{
    AUTHOR = {Kummer, E. E.},
     TITLE = {\"{U}ber die {E}rg\"{a}nzungss\"{a}tze zu den allgemeinen
              {R}eciprocit\"{a}tsgesetzen},
   JOURNAL = {J. Reine Angew. Math.},
  FJOURNAL = {Journal f\"{u}r die Reine und Angewandte Mathematik. [Crelle's
              Journal]},
    VOLUME = {44},
      YEAR = {1852},
     PAGES = {93--146},
      ISSN = {0075-4102},
   MRCLASS = {DML},
  MRNUMBER = {1578793},
       DOI = {10.1515/crll.1852.44.93},
       URL = {https://doi.org/10.1515/crll.1852.44.93},
}

\bib{LiRu}{article}{
   author={Li, Wen-Ching Winnie},
   author={Rudnick, Zeev},
   title={Pair arithmetical equivalence for quadratic fields},
   journal={Math. Z.},
   volume={299},
   date={2021},
   number={1-2},
   pages={797--826},
   issn={0025-5874},
   review={\MR{4311619}},
   doi={10.1007/s00209-021-02706-w},
}

\bib{manin}{article}{
    AUTHOR = {Manin, Ju. I.},
     TITLE = {The {H}asse-{W}itt matrix of an algebraic curve},
   JOURNAL = {Izv. Akad. Nauk SSSR Ser. Mat.},
  FJOURNAL = {Izvestiya Akademii Nauk SSSR. Seriya Matematicheskaya},
    VOLUME = {25},
      YEAR = {1961},
     PAGES = {153--172},
      ISSN = {0373-2436},
}

\bib{MS}{article}{
   author={Kumar Murty, Vijaya},
   author={Scherk, John},
   title={Effective versions of the Chebotarev density theorem for function
   fields},
   journal={C. R. Acad. Sci. Paris S\'{e}r. I Math.},
   volume={319},
   date={1994},
   number={6},
   pages={523--528},
}

\bib{neukirch}{article}{
    AUTHOR = {Neukirch, J\"{u}rgen},
     TITLE = {Kennzeichnung der {$p$}-adischen und der endlichen
              algebraischen {Z}ahlk\"{o}rper},
   JOURNAL = {Invent. Math.},
  FJOURNAL = {Inventiones Mathematicae},
    VOLUME = {6},
      YEAR = {1969},
     PAGES = {296--314},
      ISSN = {0020-9910},
}
\bib{Pi}{thesis}{
AUTHOR = {Pintonello, M.},
TITLE = {Characterizing number fields with quadratic L-functions},
NOTE={ALGANT Master Thesis in Mathematics,
Università degli studi di Padova \& Universiteit Leiden, 25 June 2018},
}

\bib{rosen}{article}{
   author={Rosen, Michael},
   title={The Hilbert class field in function fields},
   journal={Exposition. Math.},
   volume={5},
   date={1987},
   number={4},
   pages={365--378},
   issn={0723-0869},
}

\bib{ruck}{article}{ 
    AUTHOR = {R\"{u}ck, H.-G.},
     TITLE = {Class groups and {$L$}-series of function fields},
   JOURNAL = {J. Number Theory},
  FJOURNAL = {Journal of Number Theory},
    VOLUME = {22},
      YEAR = {1986},
    NUMBER = {2},
     PAGES = {177--189},
      ISSN = {0022-314X},
}

\bib{Schmidt1931}{article}{
author = {Schmidt, F. K.},
journal = {Mathematische Zeitschrift},
pages = {1--32},
title = {Analytische Zahlentheorie in K\"orpern der Charakteristik $p$},
volume = {33},
year = {1931},
}

\bib{Serre}{book}{
AUTHOR = {Serre, Jean-Pierre},
     TITLE = {Algebraic groups and class fields},
    SERIES = {Graduate Texts in Mathematics},
    VOLUME = {117},
      PUBLISHER = {Springer-Verlag, New York},
      YEAR = {1988},
     PAGES = {x+207}
     }

\bib{solomatin}{article}{
author = {Solomatin, P.},
title= {On Artin L-functions and Gassmann equivalence for global function fields},
note = {Preprint (2016). \url{https://arxiv.org/abs/1610.05600}},
}


\bib{StV}{article}{
author={St\"{o}hr, Karl-Otto},
author={Voloch, Jos\'e Felipe},
     TITLE = {A formula for the {C}artier operator on plane algebraic
              curves},
   JOURNAL = {J. Reine Angew. Math.},
     VOLUME = {377},
      YEAR = {1987},
     PAGES = {49--64},
}

\bib{SV}{article}{
author={Sutherland, Andrew V.},
author={Voloch, Jos\'e Felipe},
title={Maps between curves and arithmetic obstructions},
note={preprint, arxiv:1709.05734, Proceedings of AGCT 16, AMS Contemporary Mathematics, to appear.},
}

\bib{tate}{article}{
    AUTHOR = {Tate, John},
     TITLE = {Endomorphisms of abelian varieties over finite fields},
   JOURNAL = {Invent. Math.},
  FJOURNAL = {Inventiones Mathematicae},
    VOLUME = {2},
      YEAR = {1966},
     PAGES = {134--144},
      ISSN = {0020-9910},
}

\bib{uchida76}{article}{
    AUTHOR = {Uchida, K\^{o}ji},
     TITLE = {Isomorphisms of {G}alois groups},
   JOURNAL = {J. Math. Soc. Japan},
  FJOURNAL = {Journal of the Mathematical Society of Japan},
    VOLUME = {28},
      YEAR = {1976},
    NUMBER = {4},
     PAGES = {617--620},
      ISSN = {0025-5645},
}

\bib{uchida77}{article}{
    AUTHOR = {Uchida, K\^{o}ji},
     TITLE = {Isomorphisms of {G}alois groups of algebraic function fields},
   JOURNAL = {Ann. of Math. (2)},
  FJOURNAL = {Annals of Mathematics. Second Series},
    VOLUME = {106},
      YEAR = {1977},
    NUMBER = {3},
     PAGES = {589--598},
      ISSN = {0003-486X},
}

%\bib{V}{article}{
%AUTHOR = {Voloch, Jos\'{e} Felipe},
%    TITLE = {An analogue of the {W}eierstrass {$\zeta$}-function in
%            characteristic {$p$}},
%  JOURNAL = {Acta Arith.},
%  VOLUME = {79},
%    YEAR = {1997},
% NUMBER = {1},
%    PAGES = {1--6},
%}

\bib{VW}{article}{
AUTHOR = {Voloch, Jos\'{e} Felipe},
  AUTHOR = {Walker, Judy L.},
     TITLE = {Euclidean weights of codes from elliptic curves over rings},
   JOURNAL = {Trans. Amer. Math. Soc.},
  %FJOURNAL = {Transactions of the American Mathematical Society},
    VOLUME = {352},
      YEAR = {2000},
    NUMBER = {11},
     PAGES = {5063--5076},
}

\bib{Z1}{article}{
AUTHOR = {Zilber, Boris},
     TITLE = {A curve and its abstract {J}acobian},
   JOURNAL = {Int. Math. Res. Notices},
      YEAR = {2014},
    NUMBER = {5},
     PAGES = {1425--1439},
}

\bib{Z2}{article}{
AUTHOR = {Zilber, Boris},
TITLE = {A curve and its abstract {J}acobian},
NOTE = {Corrected version of \cite{Z1}, preprint, http://people.maths.ox.ac.uk/zilber/JacobianCor.pdf},
YEAR = {2017},
}



\end{biblist}
\end{bibdiv}
\end{document}